\documentclass[letterpaper,11pt,twoside,keywordsasfootnote,addressatend,noinfoline]{article}
\usepackage{fullpage}
\usepackage[english]{babel}
\usepackage{amssymb}
\usepackage{amsmath}
\usepackage{theorem}
\usepackage{epsfig}
\usepackage{subfigure}
\usepackage{imsart}
\usepackage{color}

\theorembodyfont{\slshape}
\newtheorem{theor}{Theorem}

\newtheorem{lemma}[theor]{Lemma}

\newenvironment{proof}{\noindent{\scshape Proof.}}{\hspace*{2mm} $\square$}

\newcommand{\Z}{\mathbb{Z}}
\newcommand{\R}{\mathbb{R}}
\newcommand{\N}{\mathbb{N}}
\newcommand{\D}{\mathbb{D}}
\newcommand{\ind}{\mathbf{1}}
\newcommand{\ep}{\epsilon}

\DeclareMathOperator{\card}{card}
\DeclareMathOperator{\sign}{sign}
\DeclareMathOperator{\binomial}{Binomial \,}
\DeclareMathOperator{\geometric}{Geometric \,}

\begin{document}
\begin{frontmatter}
\title     {Fluctuation versus fixation in the one-dimensional \\ constrained voter model}
\runtitle  {Constrained voter model}
\author    {Nicolas Lanchier\thanks{Research partially supported by NSF Grant DMS-10-05282} and
            Stylianos Scarlatos}
\runauthor {N. Lanchier and S. Scarlatos}
\address   {School of Mathematical and Statistical Sciences \\ Arizona State University \\ Tempe, AZ 85287, USA.}
\address   {Department of Mathematics \\ University of Patras \\ Patras 26500, Greece.}

\maketitle

\begin{abstract} \ \
 The constrained voter model describes the dynamics of opinions in a population of individuals located on a connected graph.
 Each agent is characterized by her opinion, where the set of opinions is represented by a finite sequence of consecutive integers,
 and each pair of neighbors, as defined by the edge set of the graph, interact at a constant rate.
 The dynamics depends on two parameters: the number of opinions denoted by $F$ and a so-called confidence threshold denoted by $\theta$.
 If the opinion distance between two interacting agents exceeds the confidence threshold then nothing happens, otherwise one of the
 two agents mimics the other one just as in the classical voter model.
 Our main result shows that the one-dimensional system starting from any product measures with a positive density of each opinion
 fluctuates and clusters if and only if $F \leq 2 \theta + 1$.
 Sufficient conditions for fixation in one dimension when the initial distribution is uniform and lower bounds for the
 probability of consensus for the process on finite connected graphs are also proved.
\end{abstract}



\end{frontmatter}


\section{Introduction}
\label{sec:intro}

\indent The constrained voter model has been originally introduced in \cite{vazquez_krapivsky_redner_2003} to understand the opinion
 dynamics in a spatially structured population of leftists, centrists and rightists.
 As in the popular voter model \cite{clifford_sudbury_1973, holley_liggett_1975}, the individuals are located on the vertex set of a
 graph and interact through the edges of the graph at a constant rate.
 However, in contrast with the classical voter model where, upon interaction, an individual adopts the opinion of her neighbor, it
 is now assumed that this imitation rule is suppressed when a leftist and a rightist interact.
 In particular, the model includes a social factor called homophily that prevents agents who disagree too much to interact. \vspace*{8pt}

\noindent{\bf Model description} --
 This paper is concerned with a natural generalization of the previous version of the constrained voter model that includes an arbitrary
 finite number $F$ of opinions and a so-called confidence threshold $\theta$.
 Having a connected graph $G := (V, E)$ representing the network of interactions, the state at time $t$ is a spatial configuration
 $$ \eta_t : V \to \{1, 2, \ldots, F \} := \hbox{opinion set}. $$
 Each individual looks at each of her neighbors at rate one that she imitates if and only if the opinion distance between the
 two neighbors is at most equal to the confidence threshold.
 Formally, the dynamics of the system is described by the Markov generator
 $$ \begin{array}{l} Lf (\eta) \ = \ \sum_{x \in V} \sum_{j = 1}^F \,\card \,\{y \sim x : \eta (y) = j \ \hbox{and} \ |\eta (y) - \eta (x)| \leq \theta \} \ [f (\eta_{x, j}) - f (\eta)] \end{array} $$
 where configuration $\eta_{x, j}$ is obtained from $\eta$ by setting
 $$ \eta_{x, j} (z) \ = \ j \ \ind \{z = x \} + \eta (z) \ \ind \{z \neq x \} $$
 and where $x \sim y$ means that the two vertices are connected by an edge.
 Note that the basic voter model and the original version of the constrained voter model including the three opinions leftist,
 centrist and rightist can be recovered from our general model as follows:
 $$ \begin{array}{rcl}
    \hbox{basic voter model \cite{clifford_sudbury_1973, holley_liggett_1975}} & = &
    \hbox{the process} \ \{\eta_t : t \geq 0 \} \ \hbox{with} \ F = 2 \ \hbox{and} \ \theta = 1 \vspace*{2pt} \\
    \hbox{constrained voter model \cite{vazquez_krapivsky_redner_2003}} & = &
    \hbox{the process} \ \{\eta_t : t \geq 0 \} \ \hbox{with} \ F = 3 \ \hbox{and} \ \theta = 1. \end{array} $$
 The main question about the constrained voter model is whether the system fluctuates and evolves to a global consensus or fixates
 in a highly fragmented configuration.
 To define this dichotomy rigorously, we say that {\bf fluctuation} occurs whenever
\begin{equation}
\label{eq:fluctuation}
 P \,(\eta_t (x) \ \hbox{changes value at arbitrary large $t$}) \ = \ 1 \quad \hbox{for all} \ x \in V
\end{equation}
 and that {\bf fixation} occurs if there exists a configuration $\eta_{\infty}$ such that
\begin{equation}
\label{eq:fixation}
 P \,(\eta_t (x) = \eta_{\infty} (x) \ \hbox{eventually in $t$}) \ = \ 1 \quad \hbox{for all} \ x \in V.
\end{equation}
 In other words, fixation means that the opinion of each individual is only updated a finite number of times,
 therefore fluctuation \eqref{eq:fluctuation} and fixation \eqref{eq:fixation} exclude each other.
 We define convergence to a global consensus mathematically as a {\bf clustering} of the system, i.e.,
\begin{equation}
\label{eq:clustering}
 \begin{array}{l} \lim_{t \to \infty} \ P \,(\eta_t (x) = \eta_t (y)) \ = \ 1 \quad \hbox{for all} \ x, y \in V. \end{array}
\end{equation}
 Note that, whenever $F \leq \theta + 1$, the process reduces to the basic voter model with $F$ instead of two different opinions
 for which the long-term behavior of the process is well known:
 the system on lattices fluctuates while the system on finite connected graphs fixates to a configuration in which all the individuals
 share the same opinion.
 In particular, the main objective of this paper is to study fluctuation and fixation in the nontrivial case when $F > \theta + 1$. \vspace*{8pt}

\begin{figure}[t]
\centering
\includegraphics[width=400pt]{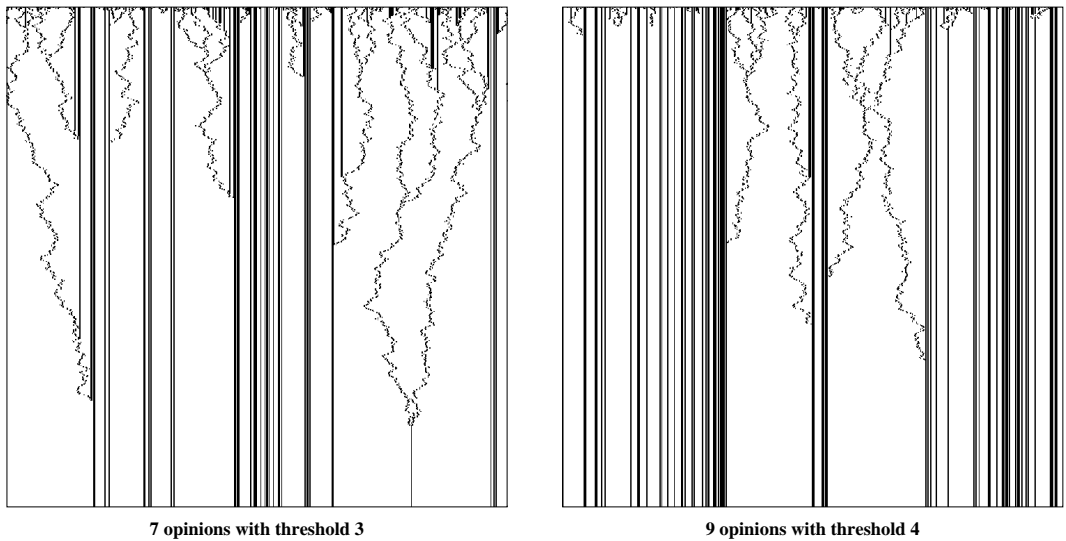}
\caption{\upshape{Two typical realizations of the constrained voter model on the torus with 600 vertices for two different pairs of parameters.
 Time goes down from time 0 to time 3000.
 The black lines represent the boundaries between the different domains, that is the edges that connect individuals with different opinions.}}
\label{fig:interface}
\end{figure}

\noindent{\bf Main results} --
 Whether the system fluctuates or fixates depends not only on the two parameters but also on the initial distribution.
 In particular, we point out that, throughout the paper, it will be assumed that the initial distribution is the product measure
 with constant densities.
 To avoid trivialities, we also assume that the initial density of each of the opinions is positive:
\begin{equation}
\label{eq:initial-assumption-1}
  \rho_j \ := \ P \,(\eta_0 (x) = j) \ > \ 0 \quad \hbox{for all} \quad x \in V \ \hbox{and} \ j \in \{1, 2, \ldots, F \}.
\end{equation}
 For the constrained voter model on the one-dimensional torus with $L$ vertices, the mean-field analysis
 in \cite{vazquez_krapivsky_redner_2003} suggests that, in the presence of three opinions and when the threshold is equal to one,
 the average domain length at equilibrium is
\begin{equation}
\label{eq:mean-field}
  L^{2 \psi (\rho_2)} \quad \hbox{where} \quad
       \psi (\rho_2) \ = \ - \frac{1}{8} \ + \ \frac{2}{\pi^2} \ \bigg[\cos^{-1} \bigg(\frac{1 - 2 \rho_2}{\sqrt 2} \bigg) \bigg]^2 \ \sim \ \frac{2 \rho_2}{\pi}
\end{equation}
 when the initial density of centrists $\rho_2$ is small and $L$ is large.
 V\'azquez et al. \cite{vazquez_krapivsky_redner_2003} also showed that these predictions agree with their numerical simulations
 from which they conclude that, when the initial density of centrists is small, the system fixates with high probability in a frozen
 mixture of leftists and rightists.
 In contrast, it is conjectured in \cite{adamopoulos_scarlatos_2012} based on an idea in \cite{lanchier_2012} that the infinite
 system fluctuates and clusters whenever $F \leq 2 \theta + 1$, which includes the threshold one model with three opinions introduced
 in \cite{vazquez_krapivsky_redner_2003}.
 To explain this apparent disagreement, we first observe that, regardless of the parameters, the system on finite
 graphs always fixate and there is a positive probability that the final configuration consists of a highly fragmented configuration,
 thus showing that spatial simulations of the necessarily finite system are not symptomatic of the behavior of its infinite counterpart.
 Our first theorem shows that the conjecture in \cite{adamopoulos_scarlatos_2012} is indeed correct.
 \begin{theor} --
\label{th:fluctuation}
 Assume \eqref{eq:initial-assumption-1} and $F \leq 2 \theta + 1$. Then,
\begin{enumerate}
 \item[a.] The process on $\Z$ fluctuates \eqref{eq:fluctuation} and clusters \eqref{eq:clustering}. \vspace*{2pt}
 \item[b.] The probability of consensus on any finite connected graph satisfies
  $$ P \,(\eta_t \equiv \hbox{constant for some} \ t > 0) \ \geq \ \rho_{F - \theta} + \rho_{F - \theta + 1} + \cdots + \rho_{\theta + 1} > 0. $$
\end{enumerate}
\end{theor}
 The intuition behind the proof is that, whenever $F \leq 2 \theta + 1$, there is a nonempty set of opinions which are within
 the confidence threshold of any other opinions.
 This simple observation implies the existence of a coupling between the constrained and basic voter models, which is the
 key to proving fluctuation.
 The proof of clustering is more difficult.
 It heavily relies on the fact that the system fluctuates but also on an analysis of the interfaces of the process through
 a coupling with a certain system of charged particles.
 In contrast, our lower bound for the probability of consensus on finite connected graphs relies on techniques from martingale theory.
 Note that this lower bound is in fact equal to the initial density of individuals who are in the confidence threshold of
 any other individuals in the system.
 Returning to the relationship between finite and infinite systems, we point out that the simulation pictures of Figure \ref{fig:interface},
 which show two typical realizations of the process on the torus under the assumptions of the theorem, suggest fixation
 of the infinite counterpart in a highly fragmented configuration, in contradiction with the first part of our theorem, showing again
 the difficulty to interpret spatial simulations.
 Note also that, for the system on the one-dimensional torus with $L$ vertices, the average domain length at equilibrium is
 bounded from below by
 $$ \card \,(V) \times P \,(\hbox{consenesus}) \ = \ L \times P \,(\hbox{consensus}) $$
 which, together with the second part of the theorem, proves that the average domain length scales like the population size when
 $F \leq 2 \theta + 1$ and that \eqref{eq:mean-field} does not hold.
 While our fluctuation-clustering result holds regardless of the initial densities provided they are all positive, whether fixation
 occurs or not seems to be very sensitive to the initial distribution.
 Also, to state our fixation results and avoid messy calculations later, we strengthen condition \eqref{eq:initial-assumption-1}
 and assume that
\begin{equation}
\label{eq:initial-assumption-2}
  \rho_1 = \rho_F > 0 \quad \hbox{and} \quad \rho_2 = \rho_3 = \cdots = \rho_{F - 1} > 0.
\end{equation}
 The next theorem looks at the fixation regime in three different contexts.
\begin{theor} --
\label{th:fixation}
 Assume \eqref{eq:initial-assumption-2}.
 Then, the process on $\Z$ fixates \eqref{eq:fixation} in the following cases:
\begin{enumerate}
 \item[a.] $F > 2 \theta + 1$ and $\rho_2 > 0$ is small enough. \vspace*{2pt}
 \item[b.] $F$ is large, $\theta / F < c_+$ and $\rho_1 = \rho_2$ where
   $$ c_+ \approx 0.21851 \ \ \hbox{is a root of} \ \ 12 \,(1 - 2X)^3 - 9X^2 (3X^2 + 4X - 6). $$
 \item[c.] $F = 4$ and $\theta = 1$ and $\rho_2 = \rho_3 < 0.2134$.
\end{enumerate}
\end{theor}
 The first part of the theorem is the converse of the first part of Theorem \ref{th:fluctuation}, thus showing that the condition
 $F = 2 \theta + 1$ is critical in the sense that
\begin{itemize}
 \item when $F \leq 2 \theta + 1$, the one-dimensional constrained voter model fluctuates when starting from any nondegenerate
  distributions \eqref{eq:initial-assumption-1} whereas \vspace*{4pt}
 \item when $F > 2 \theta + 1$, the one-dimensional constrained voter model can fixate even when starting from a nondegenerate
  distribution \eqref{eq:initial-assumption-1}.
\end{itemize}
 The last two parts of the theorem specialize in two particular cases.
 The first one looks at uniform initial distributions in which all the opinions are equally likely.
 For simplicity, our statement focuses on the fixation region when the parameters are large but our proof is not limited to large
 parameters and implies more generally that the system fixates for all pairs of parameters corresponding to the set of white dots in
 the phase diagram of Figure \ref{fig:diagram} for the one-dimensional system with up to twenty opinions.
 Note that the picture suggests that the process starting from a uniform initial distribution fixates whenever $\theta / F < c_+$ even for a
 small number of opinions.
 The second particular case returns to the slightly more general initial distributions \eqref{eq:initial-assumption-2} but focuses
 on the threshold one model with four opinions for which fixation is proved when $\rho_2$ is only slightly less than one over the number
 of opinions = 0.25.
 This last result suggests that the constrained voter model with four opinions and threshold one fixates when starting from the
 uniform product measure, although the calculations become too tedious to indeed obtain fixation when starting from this distribution. \vspace*{8pt}

\begin{figure}[t]
\centering
\scalebox{0.42}{\input{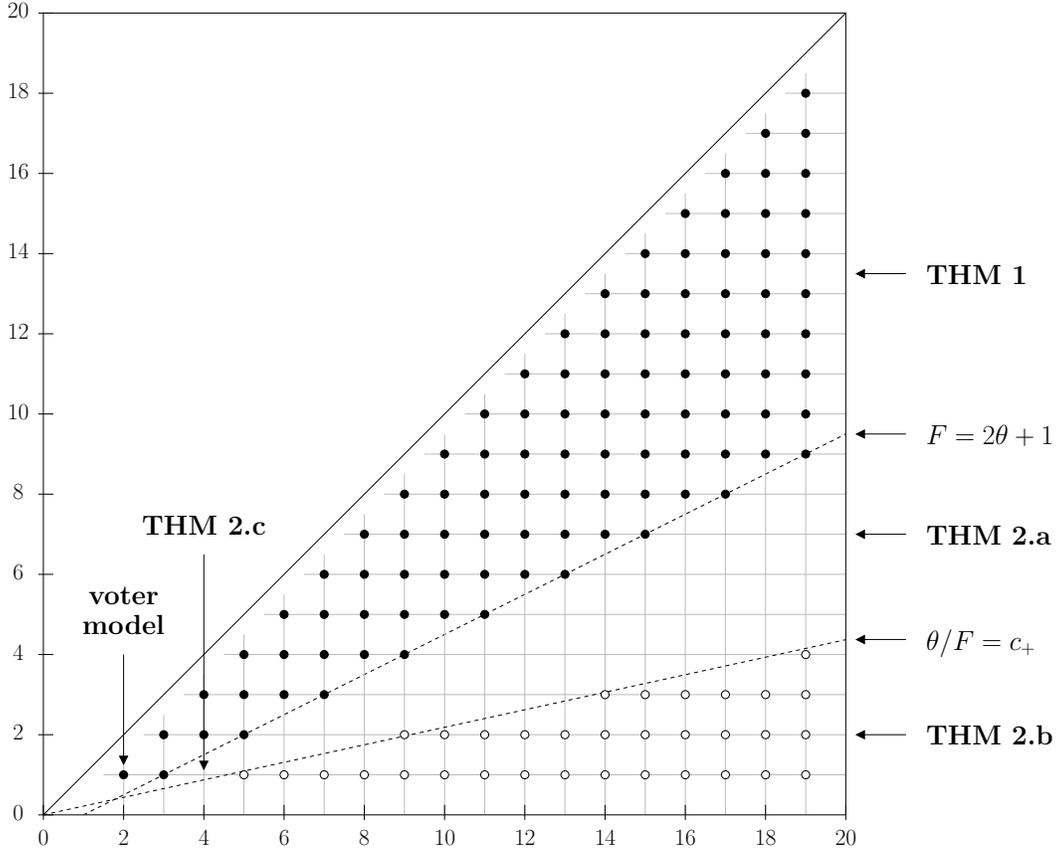}}
\caption{\upshape{Phase diagram of the one-dimensional constrained voter model in the $F - \theta$ plane along with a summary
 of our theorems.
 The black dots correspond to the set of parameters for which fluctuation and clustering are proved whereas the white dots
 correspond to the set of parameters for which fixation is proved.}}
\label{fig:diagram}
\end{figure}

\noindent{\bf Structure of the paper} --
 The rest of the article is devoted to the proof of both theorems.
 Even though our proof of fluctuation-clustering and fixation differ significantly, a common technique we introduce
 to study these two aspects for the one-dimensional process is a coupling with a certain system of charged particles that keeps
 track of the discrepancies along the edges of the graph rather than the actual opinion at each vertex.
 In contrast, our approach to analyze the process on finite connected graphs is to look at the opinion at each vertex and use,
 among other things, the optimal stopping theorem for martingales.
 The coupling with the system of charged particles is introduced in section~\ref{sec:coupling} and then used in
 section~\ref{sec:fluctuation} to prove Theorem~\ref{th:fluctuation}.
 The proof of Theorem \ref{th:fixation} is more complicated and carried out in the last five
 sections~\ref{sec:condition}--\ref{sec:fixation-particular}.
 In addition to the coupling with the system of charged particles introduced in the next section, the proof relies on a
 characterization of fixation based on so-called active paths proved in section~\ref{sec:condition} and large deviation
 estimates for the number of changeovers in a sequence of independent coin flips proved in section~\ref{sec:deviation}.


\section{Coupling with a system of charged particles}
\label{sec:coupling}

\indent To study the one-dimensional system, it is convenient to construct the process from a graphical representation and
 to introduce a coupling between the process and a certain system of charged particles that keeps track of the discrepancies
 along the edges of the lattice rather than the opinion at each vertex.
 This system of charged particles can also be constructed from the same graphical representation.
 Since the constrained voter model on general finite graphs will be studied using other techniques, we only define the graphical
 representation for the process on $\Z$, which consists of the following collection of independent Poisson processes:
\begin{itemize}
 \item for each $x \in \Z$, we let $(N_t (x, x \pm 1) : t \geq 0)$ be a rate one Poisson process, \vspace*{4pt}
 \item we denote by $T_n (x, x \pm 1) := \inf \,\{t : N_t (x, x \pm 1) = n \}$ its $n$th arrival time.
\end{itemize}
 This collection of independent Poisson processes is then turned into a percolation structure by drawing an arrow $x \to x \pm 1$
 at time $t := T_n (x, x \pm 1)$ and, given a configuration of the one-dimensional system at time $t-$, we say that this arrow
 is {\bf active} if and only if
 $$ |\eta_{t-} (x) - \eta_{t-} (x \pm 1)| \ \leq \ \theta. $$
 The configuration at time $t$ is then obtained by setting
\begin{equation}
\label{eq:rule}
  \begin{array}{rcll}
   \eta_t (x \pm 1) & = & \eta_{t-} (x) & \hbox{when the arrow $x \to x \pm 1$ is active} \vspace*{2pt} \\
                    & = & \eta_{t-} (x \pm 1) & \hbox{when the arrow $x \to x \pm 1$ is not active}. \end{array}
\end{equation}
 An argument due to Harris \cite{harris_1972} implies that the constrained voter model starting from any configuration can indeed
 be constructed using this percolation structure and rule \eqref{eq:rule}.
 From the collection of active arrows, we construct active paths as in percolation theory.
 More precisely, we say that there is an {\bf active path} from $(z, s)$ to $(x, t)$, and write $(z, s) \leadsto (x, t)$, whenever there exist
 $$ s_0 = s < s_1 < \cdots < s_{n + 1} = t \qquad \hbox{and} \qquad
    x_0 = z, \,x_1, \,\ldots, \,x_n = x $$
 such that the following two conditions hold:
\begin{enumerate}
 \item For $j = 1, 2, \ldots, n$, there is an active arrow from $x_{j - 1}$ to $x_j$ at time $s_j$. \vspace*{2pt}
 \item For $j = 0, 1, \ldots, n$, there is no active arrow that points at $\{x_j \} \times (s_j, s_{j + 1})$.
\end{enumerate}
 Note that conditions 1 and 2 above imply that
 $$ \hbox{for all} \ (x, t) \in \Z \times \R_+ \ \hbox{there is a unique} \ z \in \Z \ \hbox{such that} \ (z, 0) \leadsto (x, t). $$
 Moreover, because of the definition of active arrows, the opinion of vertex $x$ at time $t$ originates and is therefore equal to the
 initial opinion of vertex $z$ so we call vertex $z$ the {\bf ancestor} of vertex $x$ at time $t$.
 One of the key ingredients to studying the one-dimensional system is to look at the following process defined on the edges:
 identifying each edge with its midpoint, we set
 $$ \xi_t (e) \ := \ \eta_t (e + 1/2) - \eta_t (e - 1/2) \quad \hbox{for all} \quad e \in \D := \Z + 1/2 $$
 and think of edge $e$ as being
\begin{itemize}
 \item empty whenever $\xi_t (e) = 0$, \vspace*{2pt}
 \item occupied by a pile of $j$ particles with positive charge whenever $\xi_t (e) = j > 0$, \vspace*{2pt}
 \item occupied by a pile of $j$ particles with negative charge whenever $\xi_t (e) = - j < 0$.
\end{itemize}
 The dynamics of the constrained voter model induces evolution rules which are again Markov on this system of charged particles.
 Assume that there is an arrow $x + 1 \to x$ at time $t$ and
 $$ \begin{array}{rcl}
    \xi_{t-} (x - 1/2) & := & \eta_{t-} (x) - \eta_{t-} (x - 1) \ = \ i \vspace*{2pt} \\
    \xi_{t-} (x + 1/2) & := & \eta_{t-} (x + 1) - \eta_{t-} (x) \ = \ j \ \geq \ 0 \end{array} $$
 indicating in particular that there is a pile of $j$ particles with positive charge at $e := x + 1/2$.
 Then, we have the following alternative:
\begin{figure}[t]
\centering
\scalebox{0.38}{\input{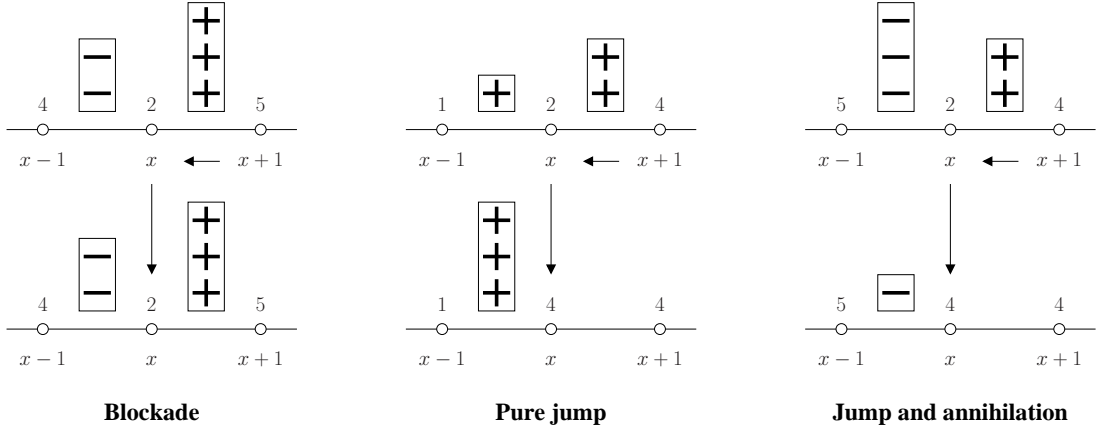}}
\caption{\upshape{Schematic illustration of the coupling between the constrained voter model and the system of charged particles
 along with their evolution rules.
 In our example, the threshold $\theta = 2$ which makes piles of three or more particles blockades with frozen particles and piles
 of two or less particles live edges with active particles.}}
\label{fig:particles}
\end{figure}
\begin{itemize}
 \item There is no particle at edge $e = x + 1/2$ or equivalently $j = 0$ in which case the individuals at vertices $x$ and $x + 1$ already agree so nothing happens. \vspace*{3pt}
 \item There is a pile of $j > \theta$ particles at edge $e = x + 1/2$ in which case $x$ and $x + 1$ disagree too much to interact so nothing happens. \vspace*{3pt}
 \item There is a pile of $j \leq \theta$ particles at $e = x + 1/2$ in which case
   $$ \begin{array}{rcl}
      \xi_t (x - 1/2) & := & \eta_t (x) - \eta_t (x - 1) \ = \ \eta_{t-} (x + 1) - \eta_{t-} (x - 1) \ = \ i + j \vspace*{2pt} \\
      \xi_t (x + 1/2) & := & \eta_t (x + 1) - \eta_t (x) \ = \ \eta_{t-} (x + 1) - \eta_{t-} (x + 1) \ = \ 0. \end{array} $$
   In particular, there is no more particles at edge $e = x + 1/2$ and a pile of $|i + j|$ particles all with the common charge $\sign (i + j)$ at edge $e - 1$.
\end{itemize}
 Similar evolution rules are obtained by exchanging the direction of the interaction or by assuming that we have $j < 0$ from which we can
 deduce the following description:
\begin{itemize}
 \item piles with more than $\theta$ particles cannot move therefore we call such piles {\bf blockades} and the particles they
   contain {\bf frozen} particles. \vspace*{3pt}
 \item piles with at most $\theta$ particles jump one step to the left or one step to the right at the same rate one therefore we call the
   particles they contain {\bf active} particles. \vspace*{3pt}
 \item when a pile with positive/negative particles jumps onto a pile with negative/positive particles, positive and negative
   particles {\bf annihilate} by pair which results in a smaller pile of particles all with the same charge.
\end{itemize}
 We refer to Figure \ref{fig:particles} for an illustration of these evolution rules.
 Note that whether an arrow is active or not can also be characterized from the state of the edge process:
 $$ x \to x \pm 1 \ \hbox{at time $t$ is active} \quad \hbox{if and only if} \quad |\xi_{t-} (x \pm 1/2)| \leq \theta. $$
 In particular, active arrows correspond to active piles of particles.

\section{Proof of Theorem \ref{th:fluctuation}}
\label{sec:fluctuation}

\indent The key ingredient to proving fluctuation of the one-dimensional system and estimating the probability of consensus
 on finite connected graphs is to partition the opinion set into two sets that we shall call the set of {\bf centrist} opinions and
 the set of {\bf extremist} opinions:
 $$ \Omega_0 \ := \ \{F - \theta, F - \theta + 1, \ldots, \theta + 1 \} \quad \hbox{and} \quad
    \Omega_1 \ := \ \{1, 2, \ldots, F \} \setminus \Omega_0. $$
 Note that the assumption $F \leq 2 \theta + 1$ implies that the set of centrist opinions is nonempty.
 Note also that both sets are characterized by the properties
\begin{equation}
\label{eq:partition}
  \begin{array}{rcll}
      j \in \Omega_0 & \hbox{if and only if} & |i - j| \leq \theta & \hbox{for all} \ i \in \{1, 2, \ldots, F \} \vspace*{4pt} \\
      j \in \Omega_1 & \hbox{if and only if} & |i - j| > \theta & \hbox{for some} \ i \in \{1, 2, \ldots, F \} \end{array}
\end{equation}
 as shown in Figure \ref{fig:partition} which gives a schematic illustration of the partition.
 Fluctuation is proved in the next lemma using this partition and relying on a coupling with the voter model.
\begin{lemma} --
\label{lem:fluctuation}
 The process on $\Z$ fluctuates whenever $F \leq 2 \theta + 1$ and $\rho_c > 0$.
\end{lemma}
\begin{proof}
 It follows from \eqref{eq:partition} that centrist agents are within the confidence threshold of every other individual.
 In particular, for each pair $(i, j) \in \Omega_0 \times \Omega_1$ we have the transition rates
\begin{equation}
\label{eq:fluctuation-1}
  \begin{array}{rcl}
     c_{i \to j} (x, \eta) & := &
          \lim_{h \to 0} \ (1/h) \,P \,(\eta_{t + h} (x) = j \,| \,\eta_t (x) = i) \vspace*{4pt} \\ & = &
          \card \,\{y \sim x : |i - j| \leq \theta \ \hbox{and} \ \eta_t (y) = j \} \ = \ \card \,\{y \sim x :  \eta_t (y) = j \} \end{array}
\end{equation}
 and similarly
\begin{equation}
\label{eq:fluctuation-2}
  \begin{array}{rcl}
     c_{j \to i} (x, \eta) & := &
          \lim_{h \to 0} \ (1/h) \,P \,(\eta_{t + h} (x) = i \,| \,\eta_t (x) = j) \vspace*{4pt} \\ & = &
          \card \,\{y \sim x : |i - j| \leq \theta \ \hbox{and} \ \eta_t (y) = i \} \ = \ \card \,\{y \sim x :  \eta_t (y) = i \}. \end{array}
\end{equation}
 Now, we introduce the process
 $$ \zeta_t (x) \ := \ \ind \,\{\eta_t (x) \in \Omega_1 \} \quad \hbox{for all} \quad x \in \Z. $$
 Since for all $j \in \Omega_1$ the transition rates $c_{i \to j} (x, \eta)$ are constant over all $i \in \Omega_0$ according
 to \eqref{eq:fluctuation-1}, we have the following local transition rate for this new process:
 $$ \begin{array}{l}
     c_{0 \to 1} (x, \zeta) \ := \
          \lim_{h \to 0} \ (1/h) \,P \,(\zeta_{t + h} (x) = 1 \,| \,\zeta_t (x) = 0) \vspace*{4pt} \\ \hspace*{25pt} = \
          \lim_{h \to 0} \ (1/h) \,\sum_{i \in \Omega_0} P \,(\zeta_{t + h} (x) = 1 \,| \,\eta_t (x) = i) \,P \,(\eta_t (x) = i \,| \,\zeta_t (x) = 0) \vspace*{4pt} \\ \hspace*{25pt} = \
          \lim_{h \to 0} \ (1/h) \,\sum_{i \in \Omega_0} \sum_{j \in \Omega_1} P \,(\eta_{t + h} (x) = j \,| \,\eta_t (x) = i) \, P \,(\eta_t (x) = i \,| \,\zeta_t (x) = 0) \vspace*{4pt} \\ \hspace*{25pt} = \
          \sum_{i \in \Omega_0} \sum_{j \in \Omega_1} c_{i \to j} (x, \eta) \,P \,(\eta_t (x) = i \,| \,\zeta_t (x) = 0) \vspace*{4pt} \\ \hspace*{25pt} = \
          \sum_{i \in \Omega_0} \sum_{j \in \Omega_1} \card \,\{y \sim x :  \eta_t (y) = j \} \,P \,(\eta_t (x) = i \,| \,\zeta_t (x) = 0) \vspace*{4pt} \\ \hspace*{25pt} = \
          \sum_{j \in \Omega_1} \card \,\{y \sim x :  \eta_t (y) = j \} \ = \ \card \,\{y \sim x : \zeta_t (y) = 1 \}. \end{array} $$
 Using \eqref{eq:fluctuation-2} in place of \eqref{eq:fluctuation-1} and some obvious symmetry, we also have
 $$ \begin{array}{l}
     c_{1 \to 0} (x, \zeta) \ := \ \card \,\{y \sim x : \eta_t (y) \in \Omega_0 \} \ = \ \card \,\{y \sim x : \zeta_t (y) = 0 \}. \end{array} $$
 This shows that the spin system $\zeta_t$ reduces to the voter model.
 In particular, the lemma directly follows from the fact that the one-dimensional voter model itself, when starting with a positive density
 of each type, fluctuates, a result proved based on duality in \cite{lanchier_2012}, pp 868--869.
\end{proof}
\begin{figure}[t]
\centering
\scalebox{0.38}{\input{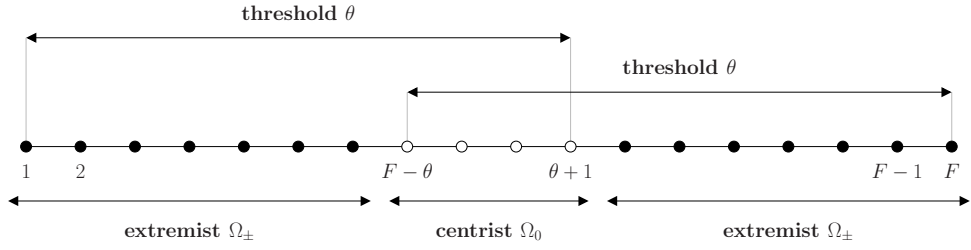}}
\caption{\upshape{Partition of the opinion set.}}
\label{fig:partition}
\end{figure}
\begin{lemma} --
\label{lem:cluster}
 The process on $\Z$ clusters whenever $F \leq 2 \theta + 1$ and $\rho_c > 0$.
\end{lemma}
\begin{proof}
 The proof strongly relies on the coupling with the voter model in the proof of the previous lemma.
 To begin with, we define the function
 $$ \begin{array}{l} u (t) \ := \ E \,(\xi_t (e)) \ = \ \sum_{j = 0}^F \, j \,P \,(\xi_t (e) = j) \end{array} $$
 which, in view of translation invariance of the initial configuration and the evolution rules, does not depend on the choice of $e$.
 Note that, since the system of charged particles coupled with the process involves deaths of particles but no births, the
 function $u (t)$ is nonincreasing in time, therefore it has a limit: $u (t) \to l$ as $t \to \infty$.
 Now, on the event that an edge $e$ is occupied by a pile of at least one particle at a given time $t$, we have the following alternative:
\begin{itemize}
 \item $e := x + 1/2$ is a blockade.
  In this case, since the centrist agents are within the confidence threshold of all the other agents, we must have
  $$ \eta_t (x) \in \Omega_1 \quad \hbox{and} \quad \eta_t (x + 1) \in \Omega_1. $$
  But since the voter model $\zeta_t$ fluctuates,
  $$ T \ := \ \inf \,\{s > t : \eta_s (x) \in \Omega_0 \ \hbox{or} \ \eta_s (x + 1) \in \Omega_0 \} \ < \ \infty \quad \hbox{almost surely}. $$
  In particular, at least of one of the frozen particles at $e$ is killed eventually. \vspace*{4pt}
 \item $e := x + 1/2$ is a live edge.
  In this case, since one-dimensional symmetric random walks are recurrent, the active pile of particles at $e$ eventually intersects
  another pile of particles, and we have the following alternative:
 \begin{itemize}
  \item The two intersecting piles of particles have opposite charge, which results in the simultaneous death of at least two particles.
  \item The two intersecting piles have the same charge and merge to form a blockade in which case we are back to the previous case:
   since the voter model $\zeta_t$ fluctuates, at least one of the frozen particles in this blockade is killed eventually.
  \item The two intersecting piles have the same charge and merge to form a larger active pile in which case the pile keeps moving until, after
   a finite number of collisions, we are back to one of the previous two possibilities:
   at least two active particles annihilate or there is creation of a blockade with at least one particle that is killed eventually.
 \end{itemize}
\end{itemize}
 In either case, as long as there are particles, there are also annihilating events indicating that the density of particles $u (t)$ is
 strictly decreasing as long as it is positive.
 In particular, the density of particles decreases to zero so there is extinction of both the active and frozen particles:
 $$ \begin{array}{l} \lim_{t \to \infty} \,P \,(\xi_t (e) \neq 0) \ = \ 0 \quad \hbox{for all} \quad e \in \Z + 1/2. \end{array} $$
 In particular, for all $x, y \in \Z$ with $x \leq y$, we have
 $$ \begin{array}{l}
      \lim_{t \to \infty} \,P \,(\eta_t (x) \neq \eta_t (y)) \ \leq \
      \lim_{t \to \infty} \,P \,(\xi_t (z + 1/2) \neq 0 \ \hbox{for some} \ x \leq z < y) \vspace*{4pt} \\ \hspace*{25pt} \leq \
      \lim_{t \to \infty} \,\sum_{z = x}^{y - 1} \,P \,(\xi_t (z + 1/2) \neq 0) \ = \ (y - x) \lim_{t \to \infty} \,P \,(\xi_t (e) \neq 0) \ = \ 0, \end{array} $$
 which proves clustering.
\end{proof} \\ \\
 The second part of the theorem, which gives a lower bound for the probability of consensus of the process on finite connected graphs,
 relies on very different techniques, namely techniques related to martingale theory following an idea from \cite{lanchier_2010}, section 3.
 However, the partition of the opinion set into centrist opinions and extremist opinions is again a key to the proof.
\begin{lemma} --
\label{lem:consensus}
 For the process on any finite connected graph,
 $$ P \,(\hbox{consensus}) \ \geq \ \rho_c \quad \hbox{whenever} \quad F \leq 2 \theta + 1. $$
\end{lemma}
\begin{proof}
 The first step is to prove that the process that keeps track of the number of supporters of any given opinion is a martingale.
 Then, applying the martingale convergence theorem and optimal stopping theorem, we obtain a lower bound for the probability of
 extinction of the extremist agents, which is also a lower bound for the probability of consensus.
 For $j = 1, 2, \ldots, F$, we set
 $$ X_t (j) \ := \ \card \,\{x \in V : \eta_t (x) = j \} \quad \hbox{and} \quad
    X_t     \ := \ \card \,\{x \in V : \eta_t (x) \in \Omega_0 \} $$
 and we observe that
\begin{equation}
\label{eq:consensus-1}
 \begin{array}{l}
   X_t \ = \ \sum_{j \in \Omega_0} X_t (j) \ = \ X_t (F - \theta) + X_t (F - \theta + 1) + \cdots + X_t (\theta + 1). \end{array}
\end{equation}
 Letting $\mathcal F_t$ denote the natural filtration of the process, we also have
 $$ \begin{array}{l}
    \lim_{h \to 0} \ (1/h) \,E \,(X_{t + h} (j) - X_t (j) \,| \,\mathcal F_t) \vspace*{4pt} \\ \hspace*{25pt} = \
    \lim_{h \to 0} \ (1/h) \,P \,(X_{t + h} (j) - X_t (j) = 1 \,| \,\mathcal F_t) \vspace*{4pt} \\ \hspace*{50pt} - \
    \lim_{h \to 0} \ (1/h) \,P \,(X_{t + h} (j) - X_t (j) = - 1 \,| \,\mathcal F_t) \vspace*{4pt} \\ \hspace*{25pt} = \
    \card \,\{(x, y) \in E : \eta_t (x) \neq j \ \hbox{and} \ \eta_t (y) = j \ \hbox{and} \ |\eta_t (x) - j| \leq \theta \} \vspace*{4pt} \\ \hspace*{50pt} - \
    \card \,\{(x, y) \in E : \eta_t (x) = j \ \hbox{and} \ \eta_t (y) \neq j \ \hbox{and} \ |\eta_t (y) - j| \leq \theta \} \ = \ 0 \end{array} $$
 indicating that the process $X_t (j)$ is a martingale with respect to the natural filtration of the constrained voter model.
 This, together with \eqref{eq:consensus-1}, implies that $X_t$ also is a martingale.
 It is also bounded because of the finiteness of the graph therefore, according to the martingale convergence theorem,
 there is almost sure convergence to a certain random variable:
 $$ X_t \ \longrightarrow \ X_{\infty} \quad \hbox{almost surely as} \ t \to \infty $$
 and we claim that $X_{\infty}$ can only take two values:
\begin{equation}
\label{eq:consensus-2}
  X_{\infty} \in \{0, N \} \quad \hbox{where} \quad \hbox{$N := \card \,(V)$ = the population size}.
\end{equation}
 To prove our claim, we note that, invoking again the finiteness of the graph, the process gets trapped in an absorbing
 state after an almost surely stopping time so we have
 $$ X_{\infty} = X_T \quad \hbox{where} \quad T := \inf \,\{t : \eta_t = \eta_s \ \hbox{for all} \ s > t \} \ \hbox{is almost surely finite}. $$
 Assuming by contradiction that $X_{\infty} = X_T \notin \{0, N \}$ gives an absorbing state with at least one centrist agent and
 at least one extremist agent.
 Since the graph is connected, this implies the existence of an edge $e = (x, y)$ such that
 $$ \eta_T (x) \in \Omega_0 \quad \hbox{and} \quad \eta_T (y) \in \Omega_1 $$
 but then we have $\eta_T (x) \neq \eta_T (y)$ and
 $$ |\eta_T (y) - \eta_T (x)| \ \leq \ \max \,((\theta + 1) - 1, F - (F - \theta)) \ = \ \theta $$
 showing that $\eta_T$ is not an absorbing state, in contradiction with the definition of time $T$.
 This proves that our claim \eqref{eq:consensus-2} is true.
 Now, applying the optimal stopping theorem to the bounded martingale $X_t$ and the almost surely finite stopping time $T$ and
 using \eqref{eq:consensus-2}, we obtain
 $$ \begin{array}{rclcl}
    E X_T & = & E X_0 & = & N \times P \,(\eta_0 (x) \in \Omega_0) \ = \ N \,\rho_c \vspace*{3pt} \\
          & = & E X_{\infty} & = & 0 \times P \,(X_{\infty} = 0) + N \times P \,(X_{\infty} = N) \ = \ N \times P \,(X_{\infty} = N), \end{array} $$
 from which it follows that
\begin{equation}
\label{eq:consensus-3}
  P \,(X_{\infty} = N) \ = \ \rho_c.
\end{equation}
 To conclude, we observe that, on the event that $X_{\infty} = N$, all the opinions present in the system after the hitting
 time $T$ are within distance $\theta$ of each other therefore the process evolves according to a voter model after that time.
 Since the only absorbing states of the voter model on finite connected graphs are the configurations in which all the agents
 share the same opinion, we deduce that the system converges to a consensus.
 This, together with \eqref{eq:consensus-3}, implies that
 $$ P \,(\hbox{consensus}) \ \geq \ P \,(X_{\infty} = N) \ = \ \rho_c. $$
 This completes the proof of the lemma.
\end{proof}

\section{Sufficient condition for fixation}
\label{sec:condition}

\indent The main objective of this section is to prove a sufficient condition for fixation of the constrained voter model based
 on certain properties of the active paths.
\begin{lemma} --
\label{lem:fixation-condition}
 For all $z \in \Z$, let
 $$ T (z) \ := \ \inf \,\{t : (z, 0) \leadsto (0, t) \}. $$
 Then, the constrained voter model fixates whenever
\begin{equation}
\label{eq:fixation-1}
 \begin{array}{l} \lim_{N \to \infty} \ P \,(T (z) < \infty \ \hbox{for some} \ z < - N) \ = \ 0. \end{array}
\end{equation}
\end{lemma}
\begin{proof}
 This is similar to the proof of Lemma 2 in \cite{bramson_griffeath_1989} and Lemma 4 in \cite{lanchier_scarlatos_2013}.
 To begin with, we define recursively a sequence of stopping times by setting
 $$ \tau_0 := 0 \quad \hbox{and} \quad \tau_j := \inf \,\{t > \tau_{j - 1} : \eta_t (0) \neq \eta_{\tau_{j - 1}} (0) \} \quad \hbox{for} \ j \geq 1. $$
 In other words, the $j$th stopping time $\tau_j$ is the $j$th time the individual at the origin changes her opinion.
 Now, we define the following random variables and collection of events:
 $$ \begin{array}{rcl}
      a_j & := & \hbox{the ancestor of vertex 0 at time $\tau_j$} \vspace*{4pt} \\
      B   & := & \{\tau_j < \infty \ \hbox{for all} \ j \} \quad \hbox{and} \ \quad
      G_N \ := \ \{|a_j| < N \ \hbox{for all} \ j \}. \end{array} $$
 The assumption \eqref{eq:fixation-1} together with reflection symmetry implies that the event $G_N$ occurs almost
 surely for some positive integer $N$, which implies that
 $$ \begin{array}{l}
     P \,(B) \ = \ P \,(B \cap (\cup_N \,G_N)) \ = \ P \,(\cup_N \,(B \cap G_N)). \end{array} $$
 Since $B$ is the event that the individual at the origin changes her opinion infinitely often, in view of the previous
 inequality, in order to establish fixation, it suffices to prove that
\begin{equation}
\label{eq:fixation-2}
 P \,(B \cap G_N) \ = \ 0 \quad \hbox{for all} \ N \geq 1.
\end{equation}
 To prove equations \eqref{eq:fixation-2}, we let
 $$ I_t (x) := \{z \in \Z : x \ \hbox{is the ancestor of $z$ at time $t$} \} \quad \hbox{and} \quad M_t (x) := \card \,(I_t (z)) $$
 be the set of descendants of $x$ at time $t$ which, due to one-dimensional nearest neighbor interactions, is necessarily an interval
 and its cardinality, respectively.
 Now, since each interaction between two individuals is equally likely to affect the opinion of each of these two individuals, the
 number of descendants of any given site is a martingale whose expected value is constantly equal to one.
 In particular, the martingale convergence theorem implies that
 $$ \begin{array}{l} \lim_{t \to \infty} \ M_t (x) \ = \ M_{\infty} (x) \quad \hbox{with probability one} \quad \hbox{where} \ E \,|M_{\infty} (x)| < \infty \end{array} $$
 therefore the number of descendants of $x$ converges to a finite value.
 Since in addition the number of descendants is an integer-valued process,
 $$ \sigma (x) \ := \ \inf \,\{t > 0 : M_t (x) = M_{\infty} (x) \} \ < \ \infty \quad \hbox{with probability one}, $$
 which further implies that, with probability one,
\begin{equation}
\label{eq:fixation-3}
 \begin{array}{l} \lim_{t \to \infty} \,I_t (x) = I_{\infty} (x) \quad \hbox{and} \quad
                  \rho (x) := \inf \,\{t > 0 : I_t (x) = I_{\infty} (x) \} < \infty. \end{array}
\end{equation}
 Finally, we note that, on the event $G_N$, the last time the individual at the origin changes her opinion is at most
 equal to the largest of the stopping times $\rho (x)$ for $x \in (- N, N)$ therefore
 $$ P \,(B \cap G_N) \ = \ P \,(\rho (x) = \infty \ \hbox{for some} \ - N < x < N) \ = \ 0 $$
 according to \eqref{eq:fixation-3}.
 This proves \eqref{eq:fixation-2} and the lemma.
\end{proof}


\section{Large deviation estimates}
\label{sec:deviation}

\indent In order to find later a good upper bound for the probability in \eqref{eq:fixation-1} and deduce a sufficient
 condition for fixation of the process, the next step is to prove large deviation estimates for the number of piles with
 $j$~particles with a given charge in a large interval.
 More precisely, the main objective of this section is to prove that for all $j$ and all $\ep > 0$ the probability that
 $$ \card \{e \in [0, N) : \xi_0 (e) = j \} \notin (1 - \ep, 1 + \ep) \ E \,(\card \{e \in [0, N) : \xi_0 (e) = j \}) $$
 decays exponentially with $N$.
 Note that, even though the initial opinions are chosen independently, the states at different edges are not independent.
 For instance, a pile of particles with a positive charge is more likely to be surrounded by negative particles.
 In particular, the result does not simply follow from large deviation estimates for the binomial distribution.
 The main ingredient is to first show large deviation estimates for the number of so-called changeovers in a sequence
 of independent coin flips.
 Consider an infinite sequence of independent coin flips such that
 $$ P \,(X_j = H) \ = \ p \quad \hbox{and} \quad P \,(X_j = T) \ = \ q = 1 - p \quad \hbox{for all} \quad j \in \N $$
 where $X_j$ is the outcome: heads or tails, at time $j$.
 We say that a {\bf changeover} occurs whenever two consecutive coin flips result in two different outcomes.
 The expected value of the number of changeovers $Z_N$ before time $N$ can be easily computed by observing that
 $$ \begin{array}{l} Z_N \ = \ \sum_{j = 0}^{N - 1} \,Y_j \quad \hbox{where} \quad Y_j \ := \ \ind \{X_{j + 1} \neq X_j \} \end{array} $$
 and by using the linearity of the expected value:
 $$ \begin{array}{l} E Z_N \ = \ \sum_{j = 0}^{N - 1} \,E Y_j \ = \ \sum_{j = 0}^{N - 1} \,P \,(X_{j + 1} \neq X_j) \ = \ N \,P \,(X_0 \neq X_1) \ = \ 2 N p \,(1 - p). \end{array} $$
 Then, we have the following large deviation estimates for the number of changeovers.
\begin{lemma} --
\label{lem:changeover}
 For all $\ep > 0$, there exists $c_0 > 0$ such that
 $$ P \,(Z_N - E Z_N \notin (- \ep N, \ep N)) \ \leq \ \exp (- c_0 N) \quad \hbox{for all $N$ sufficiently large}. $$
\end{lemma}
\begin{proof}
 To begin with, we let $\tau_{2K}$ be the time to the $2K$th changeover and notice that, since all the outcomes between two consecutive
 changeovers are identical, the sequence of coin flips up to this stopping time can be decomposed into $2K$ strings with an alternation
 of strings with only heads and strings with only tails followed by one more coin flip.
 In addition, since the coin flips are independent, the length distribution of each string is
 $$ \begin{array}{rcl}
      H_j & := & \hbox{length of the $j$th string of heads} \ = \ \geometric (q) \vspace*{2pt} \\
      T_j & := & \hbox{length of the $j$th string of tails} \ = \ \geometric (p) \end{array} $$
 and lengths are independent.
 In particular, $\tau_{2K}$ is equal in distribution to the sum of $2K$ independent geometric random variables with parameters $p$ and $q$,
 namely, we have
\begin{equation}
\label{eq:changeover-1}
  P \,(\tau_{2K} = n) \ = \ P \,(H_1 + T_1 + \cdots + H_K + T_K = n) \quad \hbox{for all} \quad n \in \N.
\end{equation}
 Now, using that, for all $K \leq n$,
 $$ \begin{array}{l}
    \displaystyle P \,(H_1 + H_2 + \cdots + H_K = n) \ = \,{n - 1 \choose K - 1} \,q^K \,(1 - q)^{n - K} \vspace*{2pt} \\ \hspace*{50pt}
    \displaystyle = \,\frac{K}{n} \ {n \choose K} \,q^K \,(1 - q)^{n - K} \ \leq \ P \,(\binomial (n, q) = K)
 \end{array} $$
 and large deviation estimates for the binomial distribution implies that
\begin{equation}
\label{eq:changeover-2}
  \begin{array}{l}
    P \,((1/K)(H_1 + H_2 + \cdots + H_K) \geq (1 + \ep)(1/q)) \vspace*{4pt} \\ \hspace*{100pt} \leq \
    P \,(\binomial ((1 + \ep)(1/q) K, q) \leq K) \ \leq \ \exp (- c_1 K) \vspace*{8pt} \\
    P \,((1/K)(H_1 + H_2 + \cdots + H_K) \leq (1 - \ep)(1/q)) \vspace*{4pt} \\ \hspace*{100pt} \leq \
    P \,(\binomial ((1 - \ep)(1/q) K, q) \geq K) \ \leq \ \exp (- c_1 K) \end{array}
\end{equation}
 for a suitable constant $c_1 > 0$ and all $N$ large. Similarly,
\begin{equation}
\label{eq:changeover-3}
  \begin{array}{l}
    P \,((1/K)(T_1 + T_2 + \cdots + T_K) \geq (1 + \ep)(1/p)) \ \leq \ \exp (- c_2 K) \vspace*{4pt} \\
    P \,((1/K)(T_1 + T_2 + \cdots + T_K) \leq (1 - \ep)(1/p)) \ \leq \ \exp (- c_2 K) \end{array}
\end{equation}
 for a suitable $c_2 > 0$ and all $N$ large.
 Combining \eqref{eq:changeover-1}--\eqref{eq:changeover-3}, we deduce that
 $$ \begin{array}{l}
     P \,((1/K) \,\tau_{2K} \notin ((1 - \ep)(1/p + 1/q), (1 + \ep)(1/p + 1/q))) \vspace*{4pt} \\ \hspace*{40pt} = \
     P \,((1/K)(H_1 + T_1 + \cdots + H_K + T_K) \notin ((1 - \ep)(1/p + 1/q), (1 + \ep)(1/p + 1/q))) \vspace*{4pt} \\ \hspace*{40pt} \leq \
     P \,((1/K)(H_1 + H_2 + \cdots + H_K) \notin ((1 - \ep)(1/q), (1 + \ep)(1/q))) \vspace*{4pt} \\ \hspace*{100pt} + \
     P \,((1/K)(T_1 + T_2 + \cdots + T_K) \notin ((1 - \ep)(1/p), (1 + \ep)(1/p))) \vspace*{4pt} \\ \hspace*{40pt} \leq \
     2 \exp (- c_1 K) + 2 \exp (- c_2 K). \end{array} $$
 Taking $K := pq N$ and observing that $pq \,(1/p + 1/q) = 1$, we deduce
 $$ \begin{array}{l}
      P \,((1/N) \,\tau_{2K} \notin (1 - \ep, 1 + \ep)) \vspace*{4pt} \\ \hspace*{40pt} = \
      P \,((1/K) \,\tau_{2K} \notin ((1 - \ep) (1/p + 1/q), (1 + \ep)(1/p + 1/q))) \ \leq \ \exp (- c_3 N) \end{array} $$
 for a suitable $c_3 > 0$ and all $N$ large.
 In particular, for all $N$ sufficiently large,
 $$ P \,((1/N) \,\tau_{2K - \ep N} \geq 1) \ \leq \ \exp (- c_4 N) \quad \hbox{and} \quad
    P \,((1/N) \,\tau_{2K + \ep N} \leq 1) \ \leq \ \exp (- c_5 N) $$
 for suitable constants $c_4 > 0$ and $c_5 > 0$ and all $N$ sufficiently large.
 Using the previous two inequalities and the fact that the event that the number of changeovers is equal to $K$ is also the event
 that the time to the $K$th changeover is less than $N$ but the time to the next changeover is more than $N$, we conclude that
 $$ \begin{array}{l}
      P \,(Z_N - E Z_N \notin (- \ep N, \ep N)) \ = \
      P \,(Z_N \notin (2 pq - \ep, 2 pq + \ep) N) \vspace*{4pt} \\ \hspace*{20pt} = \
      P \,((1/N) \,Z_N \notin (2 pq - \ep, 2 pq + \ep)) \vspace*{4pt} \\ \hspace*{20pt} = \
      P \,((1/N) \,Z_N \leq 2 pq - \ep) + P \,((1/N) \,Z_N \geq 2 pq + \ep) \vspace*{4pt} \\ \hspace*{20pt} = \
      P \,((1/N) \,\tau_{2K - \ep N} \geq 1) + P \,((1/N) \,\tau_{2K + \ep N} \leq 1) \ \leq \
        \exp (- c_4 N) + \exp (- c_5 N) \end{array} $$
 for all $N$ sufficiently large.
 This completes the proof.
\end{proof} \\ \\
 Now, we say that an edge is of type $i \to j$ if it connects an individual with initial opinion $i$ on the left to an individual
 with initial opinion $j$ on the right, and let
 $$ e_N (i \to j) \ := \ \card \,\{x \in [0, N) : \eta_0 (x) = i \ \hbox{and} \ \eta_0 (x + 1) = j \} $$
 denote the number of edges of type $i \to j$ in the interval $I_N := [0, N)$.
 Using the large deviation estimates for the number of changeovers established in the previous lemma, we can now deduce large
 deviation estimates for the number of edges of each type.
\begin{lemma} --
\label{lem:edge}
 For all $\ep > 0$, there exists $c_6 > 0$ such that
 $$ P \,(e_N (i \to j) - N \rho_i \,\rho_j \notin (- \ep N, \ep N)) \ \leq \ \exp (- c_6 N) \quad \hbox{for all $N$ large and all $i \neq j$}. $$ 
\end{lemma}
\begin{proof}
 For any given $i$, the number of edges $i \to j$ and $j \to i$  with $j \neq i$ has the same distribution as the number of changeovers in a
 sequence of independent coin flips of a coin that lands on heads with probability $\rho_i$.
 In particular, applying Lemma \ref{lem:changeover} with $p = \rho_i$ gives
\begin{equation}
\label{eq:edge-1}
  \begin{array}{l} P \,(\sum_{j \neq i} \,e_N (i \to j) - N \rho_i \,(1 - \rho_i) \notin (- \ep N, \ep N)) \ \leq \ \exp (- c_0 N) \end{array}
\end{equation}
 for all $N$ sufficiently large.
 In addition, since each $i$ preceding a changeover is independently followed by any of the remaining $F - 1$ opinions,
\begin{equation}
\label{eq:edge-2}
  \begin{array}{l} e_N (i \to j) = \binomial (K, \rho_j \,(1 - \rho_i)^{-1}) \quad \hbox{on the event} \quad \sum_{k \neq i} \,e_N (i \to k) = K. \end{array}
\end{equation}
 Combining \eqref{eq:edge-1}--\eqref{eq:edge-2} with large deviation estimates for the binomial distribution, conditioning on the number
 of edges of type $i \to k$ for some $k \neq i$, and using that
 $$ (N (1/F)(1 - 1/F) + \ep N)(F - 1)^{-1} \ = \ N/F^2 + \ep N (F - 1)^{-1} $$
 $$ (N \rho_i \,(1 - \rho_i) + \ep N) \,\rho_j \,(1 - \rho_i)^{-1} \ = \ N \rho_i \,\rho_j + \ep N \rho_j \,(1 - \rho_i)^{-1} $$
 we deduce the existence of $c_7 > 0$ such that
\begin{equation}
\label{eq:edge-3}
  \begin{array}{l}
    P \,(e_N (i \to j) - N \rho_i \,\rho_j \geq 2 \ep N) \vspace*{4pt} \\ \hspace*{20pt} \leq \
    P \,(\sum_{k \neq i} \,e_N (i \to k) - N \rho_i \,(1 - \rho_i) \geq \ep N) \vspace*{4pt} \\ \hspace*{50pt} + \
    P \,(e_N (i \to j) \geq N \rho_i \,\rho_j + 2 \ep N \ | \ \sum_{k \neq i} \,e_N (i \to k) - N \rho_i \,(1 - \rho_i) < \ep N) \vspace*{4pt} \\ \hspace*{20pt} \leq \
      \exp (- c_0 N) + P \,(\binomial (N \rho_i \,(1 - \rho_i) + \ep N, \rho_j \,(1 - \rho_i)^{-1}) \geq N \rho_i \,\rho_j + 2 \ep N) \vspace*{4pt} \\ \hspace*{20pt} \leq \
      \exp (- c_0 N) + \exp (- c_7 N) \end{array}
\end{equation}
 for all $N$ large.
 Similarly, there exists $c_8 > 0$ such that
\begin{equation}
\label{eq:edge-4}
  P \,(e_N (i \to j) - N \rho_i \,\rho_j \leq - 2 \ep N) \ \leq \ \exp (- c_0 N) + \exp (- c_8 N)
\end{equation}
 for all $N$ large.
 The lemma follows from \eqref{eq:edge-3}--\eqref{eq:edge-4}.
\end{proof} \\ \\
 Note that the large deviation estimates for the initial number of piles of particles easily follows from the previous lemma.
 Finally, from the large deviation estimates for the number of edges of each type, we deduce the analog for a general class of weight
 functions that will be used in the next section to find a sufficient condition for fixation of the constrained voter model.
\begin{lemma} --
\label{lem:weight}
 Let $\phi : E \to \R$ and assume that
 $$ \phi (e) \ := \ w (i, j) \quad \hbox{whenever} \quad \hbox{edge $e \in E$ is of type $i \to j$} $$
 with $w (i, j) = 0$ for $i = j$.
 For all $\ep > 0$, there exists $c_9 > 0$ such that
 $$ \begin{array}{l} P \,(\sum_{e \subset I_N} \,(\phi (e) - E \phi (e)) \notin (- \ep N, \ep N)) \ \leq \ \exp (- c_9 N) \quad \hbox{for all $N$ large}. \end{array} $$ 
\end{lemma}
\begin{proof}
 First, we observe that
 $$ \begin{array}{l}
     \sum_{e \subset I_N} \,(\phi (e) - E \phi (e)) \ = \
     \sum_{e \subset I_N} \phi (e) - N E \phi (e) \vspace*{4pt} \\ \hspace*{40pt} = \
     \sum_{i \neq j} \,w (i, j) \,e_N (i \to j) - N \,\sum_{i \neq j} \,w (i, j) \,P \,(e \ \hbox{is of type} \ i \to j) \vspace*{4pt} \\ \hspace*{40pt} = \
     \sum_{i \neq j} \,w (i, j) \,(e_N (i \to j) - N \rho_i \,\rho_j). \end{array} $$
 This, together with Lemma \ref{lem:edge}, implies that
 $$ \begin{array}{l}
      P \,(\sum_{e \subset I_N} \,(\phi (e) - E \phi (e)) \notin (- \ep N, \ep N)) \vspace*{4pt} \\ \hspace*{40pt} = \
      P \,(\sum_{i \neq j} \,w (i, j) \,(e_N (i \to j) - N \rho_i \,\rho_j) \notin (- \ep N, \ep N)) \vspace*{4pt} \\ \hspace*{40pt} \leq \
      P \,(w (i, j) \,(e_N (i \to j) - N \rho_i \,\rho_j) \notin (- \ep N/F^2, \ep N/F^2) \ \hbox{for some} \ i \neq j) \vspace*{4pt} \\ \hspace*{40pt} \leq \
      F^2 \,\exp (- c_{10} N) \end{array} $$ 
 for a suitable constant $c_{10} > 0$ and all $N$ large.
\end{proof}


\section{Proof of Theorem \ref{th:fixation}.a}
\label{sec:fixation}

\indent In view of Lemma \ref{lem:fixation-condition}, in order to prove fixation, it suffices to show that the probability of the
 event in equation \eqref{eq:fixation-1}, that we denote by $H_N$, tends to zero as $N \to \infty$.
 Let $\tau$ be the first time an active path starting from $(- \infty, - N)$ hits the origin, and observe that
 $$ \tau \ = \ \inf \,\{T (z) : z \in (- \infty, - N) \} \quad \hbox{where} \quad T (z) := \inf \,\{t : (z, 0) \leadsto (0, t) \} $$
 from which it follows that
 $$ H_N \ := \ \{T (z) < \infty \ \hbox{for some} \ z < - N \} \ = \ \{\tau < \infty \}. $$
 Denote by $z_- < - N$ the initial position of this active path and by $z_+ \geq 0$ the rightmost source of an active path that
 reaches the origin by time $\tau$, i.e.,
\begin{equation}
\label{eq:paths}
  \begin{array}{rcl}
    z_- & := & \min \,\{z \in \Z : (z, 0) \leadsto (0, \tau) \} \ < \ - N \vspace*{2pt} \\
    z_+ & := & \max \,\{z \in \Z : (z, 0) \leadsto (0, \sigma) \ \hbox{for some} \ \sigma < \tau \} \ \geq \ 0, \end{array}
\end{equation}
 and define $I = (z_-, z_+)$.
 Now, note that each blockade initially in the interval $I$ must have been destroyed, i.e., turned into a set of
 active particles through the annihilation of part of the particles that constitute the blockade, by time $\tau$.
 Moreover, all the active particles initially outside the interval $I$ cannot jump inside the space-time region delimited
 by the two active paths implicitly defined in \eqref{eq:paths} because the existence of such particles would contradict either
 the minimality of $z_-$ or the maximality of $z_+$.
 In particular, on the event $H_N$, all the blockades initially in $I$ must have been destroyed before time $\tau$ by either
 active particles initially in $I$ or active particles resulting from the destruction of the blockades initially in $I$.
 To estimate the probability of this last event, we first give a weight of $-1$ to each particle initially active by setting
 $$ \phi (e) \ := \ -j \quad \hbox{whenever} \quad |\xi_0 (e)| = j \leq \theta. $$
 Now, since each blockade with $j$ frozen particles can induce the annihilation of at least $j - \theta$ active particles after which
 there is a set of at most $\theta$ initially frozen particles becoming active, the weight of such an edge is set to $j - 2 \theta$, i.e.,
 $$ \phi (e) \ := \ j - 2 \theta \quad \hbox{whenever} \quad |\xi_0 (e)| = j > \theta. $$
 The fact that the event $H_N$ occurs only if all the blockades initially in the interval $I$ are destroyed by either active particles
 initially in $I$ or active particles resulting from the destruction of the blockades initially in $I$ can be expressed as
\begin{equation}
\label{eq:inclusion}
 \begin{array}{rcl}
    H_N & \subset & \big\{\sum_{e \in I} \phi (e) \leq 0 \big\} \vspace*{8pt} \\
        & \subset & \big\{\sum_{e = l}^r \phi (e) \leq 0 \ \hbox{for some $l < - N$ and some $r \geq 0$} \big\}. \end{array}
\end{equation}
 To find an upper bound for the probability of the event on the right-hand side of \eqref{eq:inclusion}, we first compute the
 expected value of the weight function $\phi$ and then use the large deviation estimates proved in Lemma \ref{lem:weight}.
 The next lemma gives an explicit expression of the expected value of the weight function and will be used repeatedly in the
 proofs of our last three theorems in order to identify sets of parameters in which fixation occurs.
 To prove this lemma as well as the so-called contribution of additional events later, we make use of the identities
\begin{equation}
\label{eq:sums}
  \begin{array}{rclcl}
    s_1 (\theta) & := & \sum_{j = 1}^{\theta} \,j   & = & (1/2) \ \theta \,(\theta + 1) \vspace*{3pt} \\
    s_2 (\theta) & := & \sum_{j = 1}^{\theta} \,j^2 & = & (1/6) \ \theta \,(\theta + 1)(2 \theta + 1) \vspace*{3pt} \\
    s_3 (\theta) & := & \sum_{j = 1}^{\theta} \,j^3 & = & (1/4) \ \theta^2 \,(\theta + 1)^2. \end{array}
\end{equation}

\begin{lemma} --
\label{lem:expected-weight}
 Assume \eqref{eq:initial-assumption-2}.
 Then, $E \phi (e) = (1/3) \,Q (\rho_1, \rho_2)$ where
 $$ \begin{array}{rcl}
     Q (X, Y) & = & - \ 6 Y \,(2 X + (F - \theta - 2) \,Y) \,\theta^2 \vspace*{4pt} \\ && \hspace*{10pt} + \
                        2 Y \,(6 X + (F - 2 \theta - 3) \,Y) \,s_1 (F - 2 \theta - 2) + 6 X^2 \,(F - 2 \theta - 1). \end{array}  $$
\end{lemma}
\begin{proof}
 To begin with, we note that
 $$ \begin{array}{rcl}
      P \,(|\xi_0 (e)| = j) & = &
      P \,(e \ \hbox{is of type} \ i \to i + j \ \hbox{for some} \ i = 1, 2, \ldots, F - j) \vspace*{4pt} \\ && \hspace*{25pt} + \
      P \,(e \ \hbox{is of type} \ i \to i - j \ \hbox{for some} \ i = j + 1, j + 2, \ldots, F) \end{array} $$
 from which we deduce that
 $$ \begin{array}{rcl}
      P \,(|\xi_0 (e)| = j) & = & 2 \,\sum_{i = 1}^{F - j} \rho_i \rho_{i + j} \vspace*{4pt} \\
                            & = & 2 \,(\rho_1 \rho_{j + 1} + \rho_{F - j} \,\rho_F) + 2 \,\sum_{i = 2}^{F - j - 1} \rho_i \,\rho_{i + j}
                            \ = \ 4 \,\rho_1 \rho_2 + 2 \,(F - j - 2) \,\rho_2^2 \end{array} $$
 for all $j < F - 1$ while
 $$ \begin{array}{l}
      P \,(|\xi_0 (e)| = j) \ = \ 2 \,\sum_{i = 1}^{F - j} \ \rho_i \rho_{i + j} \ = \ 2 \,\rho_1 \rho_F \ = \ 2 \,\rho_1^2 \end{array} $$
 for $j = F - 1$. It follows that
\begin{equation}
\label{eq:expected-1}
  \begin{array}{rcl}
    E \phi (e) & = &
      \sum_{j = 0}^{\theta} \ (-j) \,P \,(|\xi_0 (e)| = j) \vspace*{4pt} \\ && \hspace{20pt} + \
      \sum_{j = \theta + 1}^{F - 1} \,(j - 2 \theta) \,P \,(|\xi_0 (e)| = j) \big) \vspace*{4pt} \\ & = &
    2 \,\rho_2 \,\sum_{j = 0}^{\theta} \,(-j)(2 \,\rho_1 + (F - j - 2) \,\rho_2) \vspace*{4pt} \\ && \hspace{20pt} + \
    2 \,\rho_2 \,\sum_{j = \theta + 1}^{F - 2} \,(j - 2 \theta)(2 \,\rho_1 + (F - j - 2) \,\rho_2) \ + \
    2 \,\rho_1^2 \,(F - 2 \theta - 1). \end{array}
\end{equation}
 Also, decomposing the second sum in \eqref{eq:expected-1} depending on whether the number of particles~$j$ is larger or smaller
 than $2 \theta$ and changing variables, we obtain
\begin{equation}
\label{eq:expected-2}
 \begin{array}{l}
   \sum_{j = \theta + 1}^{F - 2} \,(j - 2 \theta)(2 \,\rho_1 + (F - j - 2) \,\rho_2) \vspace*{4pt} \\ \hspace*{40pt} = \
   \sum_{j = 0}^{\theta - 1} \,(- j)(2 \,\rho_1 + (F - 2 \theta - 2 + j) \,\rho_2) \vspace*{4pt} \\ \hspace*{70pt} + \ \sum_{j = 0}^{F - 2 \theta - 2} \,j \,(2 \,\rho_1 + (F - 2 \theta - 2 - j) \,\rho_2). \end{array}
\end{equation}
 Combining \eqref{eq:expected-1}--\eqref{eq:expected-2}, we obtain
 $$ \begin{array}{rcl}
   E \phi (e) & = &
   4 \,\rho_2 \,\sum_{j = 0}^{\theta - 1} \,(- j)(2 \,\rho_1 + (F - \theta - 2) \,\rho_2) \ - \ 2 \,\rho_2 \,\theta \,(2 \,\rho_1 + (F - \theta - 2) \,\rho_2) \vspace*{4pt} \\ && \hspace*{0pt} + \
   2 \,\rho_2 \,\sum_{j = 0}^{F - 2 \theta - 2} \,((2 \,\rho_1 + (F - 2 \theta - 2) \,\rho_2) \,j - \rho_2 \,j^2) \ + \ 2 \,\rho_1^2 \,(F - 2 \theta - 1) \end{array} $$
 which, recalling the notation in \eqref{eq:sums}, becomes
\begin{equation}
\label{eq:expected-3}
 \begin{array}{rcl}
   E \phi (e) & = & - \
   4 \,\rho_2 \,(2 \,\rho_1 + (F - \theta - 2) \,\rho_2) \,s_1 (\theta - 1) \ - \ 2 \,\rho_2 \,\theta \,(2 \,\rho_1 + (F - \theta - 2) \,\rho_2) \vspace*{4pt} \\ && \hspace*{10pt} + \
   2 \,\rho_2 \,(2 \,\rho_1 + (F - 2 \theta - 2) \,\rho_2) \,s_1 (F - 2 \theta - 2) \vspace*{4pt} \\ && \hspace*{10pt} - \
   2 \,\rho_2^2 \,s_2 (F - 2 \theta - 2) \ + \ 2 \,\rho_1^2 \,(F - 2 \theta - 1). \end{array}
\end{equation}
 To further simplify the expected value, we note that
\begin{equation}
\label{eq:expected-4}
 \begin{array}{rcl}
   2 s_1 (\theta - 1) + \theta & = & (\theta - 1) \,\theta + \theta \ = \ \theta^2 \vspace*{4pt} \\
        s_2 (F - 2 \theta - 2) & = & (1/3)(2F - 4 \theta - 3) \,s_1 (F - 2 \theta - 2).
 \end{array}
\end{equation}
 Then, plugging \eqref{eq:expected-4} into \eqref{eq:expected-3}, we get
 $$ \begin{array}{l}
   3 \times E \phi (e) \ = \
   - \ 6 \,\rho_2 \,(2 \,\rho_1 + (F - \theta - 2) \,\rho_2) \,\theta^2 \vspace*{4pt} \\ \hspace*{20pt} + \
   2 \,\rho_2 \,(6 \,\rho_1 + (F - 2 \theta - 3) \,\rho_2) \,s_1 (F - 2 \theta - 2) \ + \
   6 \,\rho_1^2 \,(F - 2 \theta - 1) \ = \ Q (\rho_1, \rho_2). \end{array} $$
 This completes the proof.
\end{proof} \\ \\
 To deduce Theorem \ref{th:fixation}.a, we observe that
 $$ Q (1/2, 0) \ = \ 6 (1/2)^2 \,(F - 2 \theta - 1) \ > \ 0 \quad \hbox{whenever} \quad F > 2 \theta + 1. $$
 Using the continuity of $Q$ together with Lemma \ref{lem:expected-weight} gives
 $$ E \phi (e) \ = \ (1/3) \,Q (\rho_1, \rho_2) \ > \ 0 \quad \hbox{whenever} \quad 2 \rho_1 + (F - 2) \,\rho_2 = 1 \ \hbox{and} \ \rho_2 > 0 \ \hbox{is small}. $$
 Then, using \eqref{eq:inclusion} and applying Lemma \ref{lem:weight} with $\ep = (1/2) \,E \phi (e)$, we deduce
 $$ \begin{array}{rcl}
      \lim_{N \to \infty} P \,(H_N) & \leq &
      \lim_{N \to \infty} P \,\big(\sum_{e = l}^r \phi (e) \leq 0 \ \ \hbox{for some $l < - N$ and $r \geq 0$} \big) \vspace*{4pt} \\ & \leq &
      \lim_{N \to \infty} \sum_{l < - N} \,\sum_{r \geq 0} \,P \,\big(\sum_{e = l}^r \phi (e) \leq (1/2)(r - l) \,E \phi (e) \big) \vspace*{4pt} \\ & \leq &
      \lim_{N \to \infty} \sum_{l < - N} \,\sum_{r \geq 0} \,\exp (- c_9 \,(r - l)) \ = \ 0. \end{array} $$
 This together with Lemma \ref{lem:fixation-condition} implies fixation.

\begin{table}[t!]
 \centering
 \scalebox{0.32}{\input{tab-fix.pstex_t}} \vspace*{5pt}
 \caption{\upshape{Fixation occurs whenever \eqref{eq:initial-assumption-2} holds and $\rho_1$ is larger than the numbers in the table.}}
\label{tab:fix}
\end{table}


\section{Proof of Theorem \ref{th:fixation}.b}
\label{sec:fixation-uniform}

\indent We now specialize in the case of a uniform initial distribution: $\rho_1 = \rho_2$, which forces the initial density of each
 opinion to be equal to $F^{-1}$.
 In this case, the expected value of the weight function reduces to the expression given in the following lemma.
\begin{lemma} --
\label{lem:expected-weight-uniform}
 Assume \eqref{eq:initial-assumption-2} with $\rho_1 = \rho_2$. Then,
 $$ E \phi (e) \ = \ (1/3) \,F^{-2} \,(- 6 \,(F - \theta) \,\theta^2 + (F - 2 \theta - 1)(F - 2 \theta)(F - 2 \theta + 1)). $$
\end{lemma}
\begin{proof}
 According to Lemma \ref{lem:expected-weight}, we have
\begin{equation}
\label{eq:expected-weight-uniform-1}
  E \phi (e) \ = \ (1/3) \,Q (F^{-1}, F^{-1}) \ = \ (1/3) \,F^{-2} \,Q (1, 1).
\end{equation}
 In other respects, using
 $$ (X + 3)(X - 2) + 6 \ = \ X^2 + X \ = \ X (X + 1) $$
 with $X = F - 2 \theta$ together with \eqref{eq:sums} gives
\begin{equation}
\label{eq:expected-weight-uniform-2}
 \begin{array}{rcl}
   Q (1, 1) & = & - \ 6 \,(F - \theta) \,\theta^2 + 2 \,(F - 2 \theta + 3) \,s_1 (F - 2 \theta - 2) + 6 \,(F - 2 \theta - 1) \vspace*{4pt} \\
            & = & - \ 6 \,(F - \theta) \,\theta^2 + (F - 2 \theta - 1)((F - 2 \theta + 3)(F - 2 \theta - 2) + 6) \vspace*{4pt} \\
            & = & - \ 6 \,(F - \theta) \,\theta^2 + (F - 2 \theta - 1)(F - 2 \theta)(F - 2 \theta + 1). \end{array}
\end{equation}
 The lemma follows from combining \eqref{eq:expected-weight-uniform-1}--\eqref{eq:expected-weight-uniform-2}.
\end{proof} \\ \\
 Letting $X := \theta / F$ and taking $F \to \infty$, the lemma implies that
 $$ \begin{array}{rcl}
      \sign \,(E \phi (e)) & = & \sign \,((1 - 2X)(1 - 2X)(1 - 2X)- 6 \,X^2 \,(1 - X)) \vspace*{4pt} \\
                           & = & \sign \,((1 - 2X)^3 - 6 \,X^2 \,(1 - X)). \end{array} $$
 In particular, $E \phi (e) > 0$ whenever $F$ is large and
 $$ \theta / F < c_- \ \ \hbox{where} \ \ c_- \approx 0.20630 \ \ \hbox{is a root of} \ \ (1 - 2X)^3 - 6 X^2 (1 - X). $$
 Using again Lemmas \ref{lem:fixation-condition} and \ref{lem:weight} as well as \eqref{eq:inclusion}, we deduce as in the previous section
 that fixation occurs under the condition above when the parameters are large enough.
 This is not exactly the assumption of our theorem.
 Note however that the weight function is defined based on a worst case scenario in which the active particles \emph{do their best}
 to destroy the blockades and to turn as many frozen particles as possible into active particles.
 To improve the lower bound for the asymptotic critical slope from $c_-$ to $c_+ > c_-$, the idea is to take into account additional
 events that eliminate some active particles and find lower bounds for their contribution defined as the number of active particles
 they eliminate times their probability.
 The events we consider are
\begin{enumerate}
 \item {\bf annihilation of active particles} due to the collision of a pile of active particles with positive charge with a pile of active particles with a negative charge, \vspace*{2pt}
 \item {\bf blockade formation} due to the collision of two piles of active particles with the same charge and total size exceeding the confidence threshold, \vspace*{2pt}
 \item {\bf blockade increase} due to the jump of a pile of active particles with a certain charge onto a blockade with the same charge.
\end{enumerate}
 More precisely, we introduce the events
 $$ A_{i, j} \ := \ \{\xi_0 (x - 1/2) = i \} \,\cap \,\{\xi_0 (x + 1/2) = j \} $$
 as well as the three events
\begin{equation}
\label{eq:events}
 \begin{array}{rcl}
   A & := & A_{i, j} \ \hbox{occurs for some $- \theta \leq i, j \leq \theta$ with $ij < 0$ and} \vspace*{0pt} \\
     &    & \hbox{one of the two active piles jumps onto the other active pile} \vspace*{6pt} \\
   B & := & A_{i, j} \ \hbox{occurs for some $- \theta \leq i, j \leq \theta$ with $ij > 0$ and $i + j > \theta$ and} \vspace*{0pt} \\
     &    & \hbox{one of the two active piles jumps onto the other active pile} \vspace*{6pt} \\
   C & := & A_{i, j} \ \hbox{occurs for some $ij > 0$ with $|i| \leq \theta$ and $|j| > \theta$ or $|i| > \theta$ and $|j| \leq \theta$} \vspace*{0pt} \\
     &    & \hbox{and the active pile jumps onto the blockade}. \end{array}
\end{equation}
 The next three lemmas give lower bounds for the contribution of these three events.
\begin{lemma} --
\label{lem:contribution-1}
 The contribution of the event $A$ satisfies
 $$ f (A) \ \geq \ (1/9) \,F^{-3} \,\theta \,(\theta + 1)(2F \,(2 \theta + 1) - 3 \theta \,(\theta + 1)). $$
\end{lemma}
\begin{proof}
 By conditioning on the possible values of $\eta_0 (x)$, we get
\begin{equation}
\label{eq:contribution-11}
 \begin{array}{rcl}
     P \,(A_{i, j}) & = & \sum_{k = 1}^F \,P \,(A_{i, j} \,| \,\eta_0 (x) = k) \,P \,(\eta_0 (x) = k) \vspace*{4pt} \\ & = &
     F^{-1} \,\sum_{k = 1}^F \,P \,(\eta_0 (x - 1) = k - i) \,P \,(\eta_0 (x + 1) = k + j) \vspace*{4pt} \\ & = &
     F^{-3} \,\sum_{k = 1}^F \,\ind \{1 \leq k - i \leq F \} \ \ind \{1 \leq k + j \leq F \} \vspace*{4pt} \\ & = &
     F^{-3} \,\sum_{k = 1}^F \,\ind \{\max (1 + i, 1 - j) \leq k \leq \min (F + i, F - j) \}. \end{array}
\end{equation}
 In particular, when $i > 0$ and $j < 0$ we have
 $$ \begin{array}{l} P \,(A_{i, j}) \ = \ F^{-3} \,\sum_{k = 1}^F \,\ind \{k \geq \max (1 + i, 1 - j) \} \ = \ F^{-3} \,(F - \max (i, -j)). \end{array} $$
 Note that on the event $A \cap A_{i, j}$ the number of particles that are eliminated is twice the size
 of the smallest of the two active piles therefore the contribution of this event is given by
\begin{equation}
\label{eq:contribution-12}
 \begin{array}{rcl}
     f (A \cap A_{i, j}) & = & 2 \min (i, -j) \ P \,(A \cap A_{i, j}) \vspace*{4pt} \\
                         & = & 2 \min (i, -j) \ P \,(A \cap A_{i, j} \,| \,A_{i, j}) \,P \,(A_{i, j}) \vspace*{4pt} \\
                         & \geq & (2/3) \,\min (i, -j) \ F^{-3} \,(F - \max (i, -j)) \end{array}
\end{equation}
 where the factor 1/3 is the probability that there is an arrow pointing at $x$ before there is an arrow pointing at $x \pm 1$.
 Using obvious symmetry, \eqref{eq:sums} and \eqref{eq:contribution-12}, we deduce that
 $$ \begin{array}{rcl}
     f (A) & = & \sum_{ij < 0} \ f (A \cap A_{i, j}) \ = \ 2 \ \sum_{i = 1}^{\theta} \ \sum_{j = 1}^{\theta} \ f (A \cap A_{i, -j}) \vspace*{4pt} \\
           & \geq & (4/3) \,F^{-3} \,\sum_{i = 1}^{\theta} \ \sum_{j = 1}^{\theta} \,\min (i, j)(F - \max (i, j)) \vspace*{4pt} \\
           & = & (4/3) \,F^{-3} \,\sum_{i = 1}^{\theta} \,i \,(F - i) + (4/3) \,F^{-3} \,\sum_{i = 1}^{\theta} \ \sum_{j = 1}^{i - 1} \,j \,(F - i) \vspace*{4pt} \\
           & = & (4/3) \,F^{-3} \,\sum_{i = 1}^{\theta} \,i \,(F - i) + (2/3) \,F^{-3} \,\sum_{i = 1}^{\theta} \,i \,(i - 1)(F - i) \vspace*{4pt} \\
           & = & (4/3) \,F^{-3} \,\sum_{i = 1}^{\theta} \,i^2 \,(F - i) \ = \
                 (1/9) \,F^{-3} \,\theta \,(\theta + 1)(2F \,(2 \theta + 1) - 3 \theta \,(\theta + 1)). \end{array} $$
 This completes the proof.
\end{proof}
\begin{lemma} --
\label{lem:contribution-2}
 The contribution of the event $B$ satisfies
 $$ f (B) \ \geq \ (1/9) \,F^{-3} \,\theta \,(\theta + 1) \,(3 \,(\theta + 1)(2F - 5 \theta - 1) + 2 \,(2 \theta + 1)(3 \theta - F + 2)). $$
\end{lemma}
\begin{proof}
 First, we note that taking $i > 0$ and $j > 0$ in equation \eqref{eq:contribution-11} gives
 $$ \begin{array}{l} P \,(A_{i, j}) \ = \ F^{-3} \,\sum_{k = 1}^F \,\ind \{1 + i \leq k \leq F - j \} \ = \ F^{-3} \,(F - i -j). \end{array} $$
 In addition, the event $B \cap A_{i, j}$ with $i + j > \theta$ replaces a set of $i + j$ active particles by a blockade of
 size $i + j$ so it induces the annihilation of at least
 $$ (i + j) + (i + j - 2 \theta) \ = \ 2 \,(i + j - \theta) $$
 active particles.
 The contribution of $B \cap A_{i, j}$ is therefore
\begin{equation}
\label{eq:contribution-21}
 \begin{array}{rcl}
     f (B \cap A_{i, j}) & = & 2 \,(i + j - \theta) \ P \,(B \cap A_{i, j}) \vspace*{4pt} \\
                         & = & 2 \,(i + j - \theta) \ P \,(B \cap A_{i, j} \,| \,A_{i, j}) \,P \,(A_{i, j}) \vspace*{4pt} \\
                         & \geq & (2/3)(i + j - \theta) \ F^{-3} \,(F - i - j). \end{array}
\end{equation}
 Using again some obvious symmetry together with \eqref{eq:sums} and \eqref{eq:contribution-21} and the fact that the contribution
 above is a function of $k := i + j$, we deduce that
 $$ \begin{array}{rcl}
     f (B) & = & \sum_{ij > 0} \ f (B \cap A_{ij}) \ = \ 2 \ \sum_{i = 1}^{\theta} \ \sum_{j = \theta + 1 - i}^{\theta} \ f (B \cap A_{i, j}) \vspace*{4pt} \\
           & \geq & (4/3) \,F^{-3} \,\sum_{i = 1}^{\theta} \ \sum_{j = \theta + 1 - i}^{\theta} \ (i + j - \theta)(F - i - j) \vspace*{4pt} \\ & = &
    (4/3) \,F^{-3} \,\sum_{k = \theta + 1}^{2 \theta} \ (2 \theta + 1 - k)(k - \theta)(F - k) \vspace*{4pt} \\ & = &
    (4/3) \,F^{-3} \,\sum_{k = 1}^{\theta} \ k \,(\theta + 1 - k)(F - 2 \theta - 1 + k) \vspace*{4pt} \\ & = &
    (4/3) \,F^{-3} \,\big((\theta + 1)(F - 2 \theta - 1) \,s_1 (\theta) \ + \ (3 \theta - F + 2) \,s_2 (\theta) \ - \ s_3 (\theta) \big) \vspace*{4pt} \\ & = &
    (1/9) \,F^{-3} \,\theta \,(\theta + 1) \,(3 \,(\theta + 1)(2F - 5 \theta - 2) + 2 \,(2 \theta + 1)(3 \theta - F + 2)). \end{array} $$
 This completes the proof.
\end{proof}
\begin{lemma} --
\label{lem:contribution-3}
 The contribution of the event $C$ satisfies
 $$ f (C) \ \geq \ (1/12) \,F^{-3} \,\theta \,(\theta + 1) \,(6 F \,(F - 2 \theta - 1) - 2 \,(2 \theta + 1)(2F - 2 \theta - 1) + 9 \theta \,(\theta + 1)). $$
\end{lemma}
\begin{proof}
 Assume that $0 < i \leq \theta$ and $j > \theta$.
 Then, \eqref{eq:contribution-11} again implies that
 $$ \begin{array}{l} P \,(A_{i, j}) \ = \ F^{-3} \,(F - i -j). \end{array} $$
 In addition, the event $C \cap A_{i, j}$ replaces an active pile of size $i$ and a blockade of size $j$ with a blockade of size $i + j$
 so it induces the annihilation of
 $$ (i + j - 2 \theta) - (- i + (j - 2 \theta)) \ = \ 2i $$
 active particles.
 The contribution of $C \cap A_{i, j}$ is therefore
\begin{equation}
\label{eq:contribution-31}
 \begin{array}{rcl}
     f (C \cap A_{i, j}) & = & 2 i \ P \,(C \cap A_{i, j}) \ = \
                               2 i \ P \,(C \cap A_{i, j} \,| \,A_{i, j}) \,P \,(A_{i, j}) \vspace*{4pt} \\
                         & \geq & (1/2) \,i \ F^{-3} \,(F - i - j) \end{array}
\end{equation}
 where we use that the conditional probability is larger than
 $$ 1/4 \ = \ P \,(x - 1 \to x \ \hbox{occurs before} \ x - 2 \to x - 1 \ \hbox{and} \ x - 1 \to x - 2 \ \hbox{and} \ x + 2 \to x + 1). $$ 
 Using symmetry together with \eqref{eq:sums} and \eqref{eq:contribution-31}, we get
 $$ \begin{array}{rcl}
     f (C) & = & 4 \ \sum_{i = 1}^{\theta} \ \sum_{j = \theta + 1}^{F - i} \ f (C \cap A_{i, j}) \vspace*{4pt} \\ & \geq &
                 2 F^{-3} \,\sum_{i = 1}^{\theta} \ \sum_{j = \theta + 1}^{F - i} \ i \,(F - i - j) \vspace*{4pt} \\ & = &
                   F^{-3} \,\sum_{i = 1}^{\theta} \ i \,((F - i)(F - 2 \theta - i - 1) + \theta (\theta + 1)) \vspace*{4pt} \\ & = &
                   F^{-3} \,(F (F - 2 \theta - 1) + \theta (\theta + 1)) \,s_1 (\theta) - F^{-3} \,(2F - 2 \theta - 1) \,s_2 (\theta) + F^{-3} \,s_3 (\theta) \vspace*{4pt} \\ & = &
          (1/12) \,F^{-3} \,\theta \,(\theta + 1) \,(6 F \,(F - 2 \theta - 1) - 2 \,(2 \theta + 1)(2F - 2 \theta - 1) + 9 \theta \,(\theta + 1)). \end{array} $$
 This completes the proof.
\end{proof} \\ \\
 Using again Lemma \ref{lem:fixation-condition} and the large deviation estimates of Lemma \ref{lem:weight}, we deduce that the
 one-dimensional constrained voter model fixates whenever
 $$ E \phi (e) \ + \ f (A) \ + \ f (B) \ + \ f (C) \ > \ 0 $$
 which, together with Lemmas \ref{lem:expected-weight}--\ref{lem:contribution-3}, gives the condition for fixation:
 $$ \begin{array}{l}
       12 F \,((F - 2 \theta - 1)(F - 2 \theta)(F - 2 \theta + 1) - 6 \,\theta^2 \,(F - \theta)) \vspace*{4pt} \\ \hspace*{20pt} + \
       4 \theta \,(\theta + 1)(2F \,(2 \theta + 1) - 3 \theta \,(\theta + 1)) \vspace*{4pt} \\ \hspace*{20pt} + \
       4 \theta \,(\theta + 1) \,(3 \,(\theta + 1)(2F - 5 \theta - 2) + 2 \,(2 \theta + 1)(3 \theta - F + 2)) \vspace*{4pt} \\ \hspace*{20pt} + \
       3 \theta \,(\theta + 1) \,(6 F \,(F - 2 \theta - 1) - 2 \,(2 \theta + 1)(2F - 2 \theta - 1) + 9 \theta \,(\theta + 1)) \ > \ 0. \end{array} $$
 Letting again $X := \theta / F$ and taking $F \to \infty$, we obtain fixation whenever
 $$ \begin{array}{l}
      12 \,((1 - 2X)^3 - 6 X^2 \,(1 - X)) + 4 X^2 \,(4X - 3X^2) \vspace*{4pt} \\ \hspace*{40pt} + \
      4X^2 \,(3X \,(2 - 5X) + 4X \,(3X - 1)) + 3X^2 \,(6 \,(1 - 2X) - 4X \,(2 - 2X) + 9X^2) \vspace*{4pt} \\ \hspace*{20pt} = \
      12 \,(1 - 2X)^3 - 9X^2 \,(3X^2 + 4X - 6) \ > \ 0. \end{array} $$
 This gives fixation when $F$ is large and
 $$ \theta / F < c_+ \ \ \hbox{where} \ \ c_+ \approx 0.21851 \ \ \hbox{is a root of} \ \ 12 (1 - 2X)^3 - 9X^2 (3X^2 + 4X - 6), $$
 which completes the proof of Theorem \ref{th:fixation}.b.


\section{Proof of Theorem \ref{th:fixation}.c}
\label{sec:fixation-particular}

\indent In this last section, we assume \eqref{eq:initial-assumption-2}, thus returning to slightly more general initial distributions,
 but specialize in the system with threshold one and four opinions, which is not covered by part b of the theorem.
 Note that, applying Lemma \ref{lem:expected-weight} with $F = 4$ and $\theta = 1$, we obtain
 $$ \begin{array}{l}
   E \phi (e) \ = \ (1/3) \,Q (\rho_1, \rho_2) \vspace*{4pt} \\ \hspace*{15pt} = \
  (1/3)(- 6 \rho_2 \,(2 \rho_1 + \rho_2) + 2 \rho_2 \,(6 \rho_1 - \rho_2) \,s_1 (0) + 6 \rho_1^2) \ = \ - 2 \rho_2 \,(2 \rho_1 + \rho_2) + 2 \rho_1^2. \end{array} $$
 Then, using that $\rho_1 + \rho_2 = 1/2$, we get
\begin{equation}
\label{eq:expected-weight}
  E \phi (e) \ = \ - 2 \,\rho_2 \,(1 - \rho_2) + 2 \,(1/2 - \rho_2)^2 \ = \ 4 \,\rho_2^2 - 4 \,\rho_2 + 1/2.
\end{equation}
 In particular, using the same arguments as in the previous two sections, we deduce that the system with threshold one and four opinions
 fixates whenever
 $$ \rho_2 \ < \ (1/4)(2 - \sqrt 2) \ \approx \ 0.1464. $$
 To improve this condition to the one stated in the theorem, we follow the same strategy as in the previous section, namely we
 compute the contribution of the three events \eqref{eq:events}.
 Using the specific value of the parameters allows us to significantly improve the lower bounds for the contribution of the three
 events.
 This is done in the following three lemmas.
\begin{lemma} --
\label{lem:contribution-4}
 Assume that $F = 4$ and $\theta = 1$. Then,
 $$ f (A) \ \geq \ (181/225) \,\rho_2^2 + 2 \,(1/2 - \rho_2)^2 \rho_2 \,(1 - (26/75) \,\rho_2  - 2 \rho_2 \,(1 - \rho_2)(5 \rho_2 + 20)^{-1}). $$
\end{lemma}
\begin{proof}
 To simplify the notation, we define
 $$ \bar \eta_t (x \to y) \ := \ (\eta_t (x), \eta_t (x + 1), \ldots, \eta_t (y)) $$
 and partition the event $A$ into two events distinguishing between two types of initial conditions that result in two different contributions:
 $$ \begin{array}{rcl}
     A_1 & := & A \hbox{ and } A_1' := \{\bar \eta_0 (x \to x + 2) = (2, 1, 2) \ \hbox{or} \ (2, 3, 2) \ \hbox{or} \ (3, 2, 3) \ \hbox{or} \ (3, 4, 3)\} \vspace*{2pt} \\
     A_2 & := & A \hbox{ and } A_2' := \{\bar \eta_0 (x \to x + 2) = (1, 2, 1) \ \hbox{or} \ (4, 3, 4)\} \ . \end{array} $$
 Also, we let $T$ be the time of the first active arrow that either starts or points at $x + 1$.
 Since two particles are eliminated on $A$, the contribution of the event $A_1$ is given by
 $$ \begin{array}{rcl}
      f (A_1) & =  & 2 \times P \,(A_1) \ = \ 2 \times P \,(A_1 \,| \,A_1') \,P \,(A_1') \vspace*{4pt} \\
              & =  & 2 \times (2 \rho_1 \rho_2^2 + 2 \rho_2^3) \times P \,(A_1 \,| \,\bar \eta_0 (x \to x + 2) = (2, 1, 2)) \vspace*{4pt} \\
              & \geq & 2 \,\rho_2^2 \times (2/4) \times P \,(\eta_t (x) = \eta_t (x + 2) = 2 \ \hbox{for all} \ t < T \,| \,\bar \eta_0 (x \to x + 2) = (2, 1, 2)) \vspace*{4pt} \\
              & \geq & \rho_2^2  \ (1 - 2 \,P \,(\eta_t (x) \in \{1, 3 \} \ \hbox{for some} \ t < T \,| \,\bar \eta_0 (x \to x + 2) = (2, 1, 2))). \end{array} $$
 Now, we let $J$ be the distance between vertex $x$ and the rightmost agent with either initial opinion~1 or initial opinion~3 to the left of $x$, i.e.,
 $$ J \ := \ \inf \{j > 0 : \eta_0 (x - j) \in \{1, 3 \} \} $$
 and observe that, in order to have a change of opinion at $x$ before time $T$, there must be a sequence of at least $J$ arrows all occurring before time $T$.
 Summing over all possible positions of this rightmost agent and using that $T$ is exponentially distributed with rate 4, we obtain
\begin{equation}
\label{eq:contribution-41}
 \begin{array}{l}
    P \,(\eta_t (x) \in \{1, 3 \} \ \hbox{for some} \ t < T \,| \,\bar \eta_0 (x \to x + 2) = (2, 1, 2)) \vspace*{4pt} \\
      \hspace*{40pt} \leq \ (\rho_1 + \rho_2)(2/6)((1/2) + (1/2)(1/5)^2) + \sum_{j = 2}^{\infty} \ (\rho_1 + \rho_2)^{j - 1} (\rho_1 + \rho_2) (1/5)^j \vspace*{4pt} \\
      \hspace*{40pt} = \ (1/6)((1/2) + (1/2)(1/5)^2) + \sum_{j = 2}^{\infty} \ (1/2)^{j - 1} (1/2) (1/5)^j \vspace*{4pt} \\
      \hspace*{40pt} = \ (1/12)(1+1/25) + (1/100) \sum_{j = 0}^{\infty} \ (1/10)^j \ = \ 22/225 \end{array}
\end{equation}
 where the first term corresponds to the case where $J=1$ and is obtained by further conditioning on whether the arrow $x \to x - 1$ and the reverse arrow occur before
 time $T$ or not.
 Combining our lower bound for the contribution of $A_1$ together with \eqref{eq:contribution-41}, we obtain
\begin{equation}
\label{eq:contribution-42}
  f (A_1) \ \geq \ \rho_2^2  \ (1 - 2 \times (22/225)) \ = \ (181/225) \,\rho_2^2.
\end{equation}
 Repeating the same reasoning for the event $A_2$ gives
$$ \begin{array}{rcl}
 f (A_2) & =  & 2 \times P \,(A_2) \ = \ 2 \times P \,(A_2 \,| \,A_2') \,P \,(A_2') \vspace*{4pt} \\
         & =  & 2 \times 2 \rho_1^2 \,\rho_2 \times P \,(A_2 \,| \,\bar \eta_0 (x \to x + 2) = (1, 2, 1)) \vspace*{4pt} \\
         & \geq & 4 \rho_1^2 \,\rho_2 \times (2/4)(1 - 2 \,P \,(\eta_t (x) = 2 \ \hbox{for some} \ t < T \,| \,\bar \eta_0 (x \to x + 2) = (1, 2, 1))). \end{array} $$
 Conditioning on all possible values of $J$ where $J$ now keeps track of the position of the closest agent with opinion 2 to the left of $x$, which is geometric
 with parameter $\rho_2$, we get
\begin{equation}
\label{eq:contribution-43}
  \begin{array}{l}
     P \,(\eta_t(x) = 2 \ \hbox{for some} \ t < T \,| \,\bar \eta_0 (x \to x + 2) = (1, 2, 1)) \vspace*{4pt} \\ \hspace*{25pt} \leq \
       \rho_2 \,(2/6)((1/2) + (1/2)(1/5)^2) + \sum_{j = 2}^{\infty} \ (2 \rho_1 + \rho_2)^{j - 1} \rho_2 \,(1/5)^j \vspace*{4pt} \\ \hspace*{25pt} = \
       \rho_2 \,(1/6)(1 + 1/25) + \sum_{j = 2}^{\infty} \ (1 - \rho_2)^{j - 1} \rho_2 \,(1/5)^j \vspace*{4pt} \\ \hspace*{25pt} = \
       (13/75) \,\rho_2 + \rho_2 \,(1 - \rho_2)(5 \rho_2 + 20)^{-1}. \end{array}
\end{equation}
 Therefore, the contribution of the event $A_2$ is bounded by
\begin{equation}
\label{eq:contribution-44}
 \begin{array}{rcl}
    f (A_2) & \geq & 2 \rho_1^2 \,\rho_2 \,(1 - (26/75) \,\rho_2  - 2 \rho_2 \,(1 - \rho_2)(5 \rho_2 + 20)^{-1}) \vspace*{4pt} \\
            & = & 2 \,(1/2 - \rho_2)^2 \rho_2 \,(1 - (26/75) \,\rho_2  - 2 \rho_2 \,(1 - \rho_2)(5 \rho_2 + 20)^{-1}). \end{array}
\end{equation}
 The lemma follows from the combination of \eqref{eq:contribution-42} and \eqref{eq:contribution-44}.
\end{proof}
\begin{lemma} --
\label{lem:contribution-5}
 Assume that $F = 4$ and $\theta = 1$. Then,
 $$ f (B) \ \geq \ 4 \,(1/2 - \rho_2) \,\rho_2^2 \ (203/225 - (13/75) \,\rho_2 + \rho_2 \,(1 - \rho_2)(5 \rho_2 + 20)^{-1}). $$
\end{lemma}
\begin{proof}
 Note that the set of initial configurations on the event $B$ is
 $$ B' \ := \ \{\bar \eta_0 (x \to x + 2) = (1, 2, 3) \ \hbox{or} \ (3, 2, 1) \ \hbox{or} \ (2, 3, 4) \ \hbox{or} \ (4, 3, 2)\}. $$
 Again, we let $T$ be the time of the first active arrow that either starts or points at $x + 1$.
 Since two active particles are eliminated on the event $B$, we have
$$ \begin{array}{rcl}
     f (B)  & =  & 2 \times P \,(B) \ = \ 2 \times P \,(B \,| \,B') \,P \,(B') \vspace*{4pt} \\
            & =  & 2 \times 4 \,\rho_1 \rho_2^2 \times P \,(B \,| \,\bar \eta_0 (x \to x + 2) = (1, 2, 3)) \vspace*{4pt} \\
            & \geq & 4 \,\rho_1 \rho_2^2  \ P \,(\bar \eta_t (x \to x + 2) = (1, 2, 3) \ \hbox{for all} \ t < T \,| \,\bar \eta_0 (x \to x + 2) = (1, 2, 3)) \vspace*{4pt} \\
            & \geq & 4 \,\rho_1 \rho_2^2 \ (1 - P \,(\eta_t (x) = 2 \ \hbox{for some} \ t < T \,| \,\bar \eta_0 (x \to x + 2) = (1, 2, 3)) \vspace*{4pt} \\
            &      & \hspace*{45pt} - \ P \,(\eta_t(x + 2) \in \{2, 4 \} \ \hbox{for some} \ t < T \,| \,\bar \eta_0 (x \to x + 2) = (1, 2, 3))). \end{array} $$
 Using \eqref{eq:contribution-41} and \eqref{eq:contribution-43} as in the previous lemma gives
 $$ f (B) \ \geq \ 4 \,(1/2 - \rho_2) \,\rho_2^2 \ (1 - 22/225 - (13/75) \,\rho_2 + \rho_2 \,(1 - \rho_2)(5 \rho_2 + 20)^{-1}). $$
 This completes the proof.
\end{proof}

\begin{figure}[t]
\centering
\scalebox{0.42}{\input{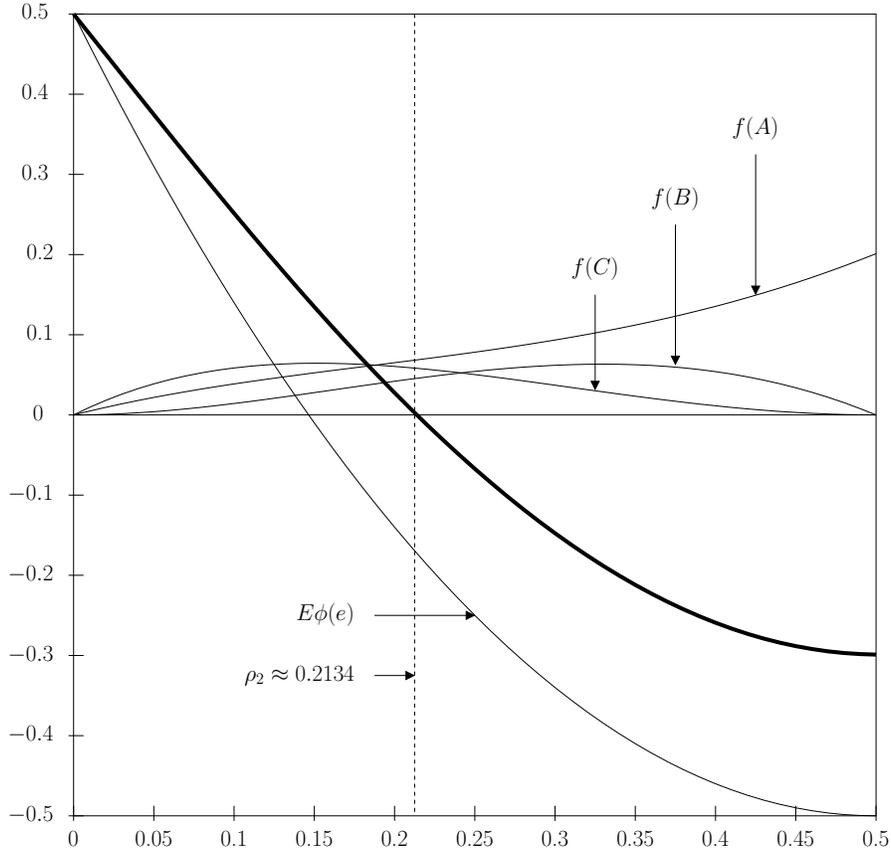}}
\caption{\upshape{Expected value of $\phi (e)$ and lower bounds for the contributions with respect to the density $\rho_2$.
 The thick black curve represents the sum of all four values, which is positive for all $\rho_2 < 0.2134$.}}
\label{fig:contribut}
\end{figure}

\begin{lemma} --
\label{lem:contribution-6}
 Assume that $F = 4$ and $\theta = 1$. Then,
 $$ f (C) \ \geq \ 4 \,(1/2 - \rho_2)^2 \,\rho_2 \ (1 - (5/9) \,\rho_2 - 2 \,\rho_2 \,(1 - \rho_2)(3 \rho_2 + 6)^{-1}). $$
\end{lemma}
\begin{proof}
 The set of initial configurations is now
 $$ C' \ := \ \{\bar \eta_0 (x \to x + 2) = (1, 3, 4) \ \hbox{or} \ (4, 3, 1) \ \hbox{or} \ (1, 2, 4) \ \hbox{or} \ (4, 2, 1)\}. $$
 Note that the time $T$ of the first active arrow that either starts or points at $x + 1$ is now exponentially distributed with rate two.
 Since two active particles are eliminated on $C$, we have
$$ \begin{array}{rcl}
     f (C) & =  & 2 \times P \,(C) \ = \ 2 \times P \,(C \,| \,C') \,P \,(C') \vspace*{4pt} \\
           & =  & 2 \times 4 \,\rho_1^2 \rho_2 \times P \,(C \,| \,\bar \eta_0 (x \to x + 2) = (1, 3, 4)) \vspace*{4pt} \\
           & \geq & 4 \,\rho_1^2 \,\rho_2 \ P \,(\eta_t (x \to x + 2) = (1, 3, 4) \ \hbox{for all} \ t < T \,| \,\bar \eta_0 (x \to x + 2) = (1, 3, 4)) \vspace*{4pt} \\
           & \geq & 4 \,\rho_1^2 \,\rho_2 \ (1 - 2 \,P \,(\eta_t (x) = 2 \ \hbox{for some} \ t < T \,| \,\bar \eta_0 (x \to x + 2) = (1, 3, 4))). \end{array} $$
 Using the same reasoning as in \eqref{eq:contribution-43} but recalling that $T$ is now exponentially distributed with parameter two instead of four, we also have
 $$ \begin{array}{l}
      P \,(\eta_t(x) = 2 \ \hbox{for some} \ t < T \,| \,\bar \eta_0 (x \to x + 2) = (1, 3, 4)) \vspace*{4pt} \\ \hspace*{25pt} \leq \
        \rho_2 \,(2/4)((1/2) + (1/2)(1/3)^2) + \sum_{j = 2}^{\infty} \ (2 \rho_1 + \rho_2)^{j - 1} \rho_2 \,(1/3)^j \vspace*{4pt} \\ \hspace*{25pt} = \
        \rho_2 \,(1/4)(1 + 1/9) + \sum_{j = 2}^{\infty} \ (1 - \rho_2)^{j - 1} \rho_2 \,(1/3)^j \vspace*{4pt} \\ \hspace*{25pt} = \
        (5/18) \,\rho_2 + \rho_2 \,(1 - \rho_2)(3 \rho_2 + 6)^{-1} \end{array} $$
 from which we deduce that
 $$ f (C) \ \geq \ 4 \,(1/2 - \rho_2)^2 \,\rho_2 \ (1 - (5/9) \,\rho_2 - 2 \,\rho_2 \,(1 - \rho_2)(3 \rho_2 + 6)^{-1}). $$
 This completes the proof.
\end{proof} \\ \\
 We refer the reader to Figure \ref{fig:contribut} for a plot of the expected value of $\phi (e)$ and the lower bounds proved
 in Lemmas \ref{lem:contribution-4}--\ref{lem:contribution-6} with respect to the initial density $\rho_2$.
 Using again the same arguments as in the previous two sections, we deduce that the one-dimensional constrained voter model with threshold
 one and four opinions fixates whenever
 $$ \begin{array}{l} (5 \rho_2 + 20)(3 \rho_2 + 6)\big(E \phi (e) + \sum_{E = A, B, C} \,f (E) \big) \ > \ 0 \end{array} $$
 which, recalling \eqref{eq:expected-weight}, using Lemmas \ref{lem:contribution-4}--\ref{lem:contribution-6}, and expanding and simplifying
 the expression above, gives the following sufficient condition for fixation: $P (\rho_2) > 0$ where
 $$ P (X) \ = \ 3000 X^6 - 8204 X^5 - 23080 X^4 + 115251 X^3 - 37635 X^2 - 39150 X + 9000. $$
 To complete the proof of the theorem, we need the following technical lemma.
\begin{lemma} --
\label{lem:polynomial}
 The polynomial $P$ is decreasing and has a unique root in $(0, 1/2)$.
\end{lemma}
\begin{proof}
 To begin with, we introduce the polynomials
 $$ \begin{array}{rcl}
      P_1 (X) & = & 18000 X - 41020 \vspace*{2pt} \\
      P_2 (X) & = & 92320 X^3 - 345753 X^2 + 75270 X + 39150 \end{array} $$
 and observe that the derivative of $P$ can be written as
\begin{equation}
\label{eq:polynomial-1}
 \begin{array}{rcl}
    P' (X) & = & 18000 X^5 - 41020 X^4 - 92320 X^3 + 345753 X^2 - 75270 X - 39150 \vspace*{2pt} \\
           & = &  X^4 \,P_1 (X) - P_2 (X). \end{array}
\end{equation}
 Now, it is clear that
\begin{equation}
\label{eq:polynomial-2}
  P_1 (X) < 0 \quad \hbox{for all} \quad X \in (0, 1/2).
\end{equation}
 In other respects, we have
 $$ P_2' (X) \ = \ 276960 X^2 - 691506 X + 75270 \quad \hbox{and} \quad P_2'' (X) \ = \ 553920 X - 691506 $$
 showing that $P_2'$ is decreasing in the interval $(0, 1/2)$.
 Therefore, the polynomial $P_2$ is concave in this interval, from which it follows that
\begin{equation}
\label{eq:polynomial-3}
 \begin{array}{l} \min_{X \in (0, 1/2)} \,P_2 (X) \ = \ \min \,(P_2 (0), P_2 (1/2)) \ = \ P_2 (1/2) \ = \ 1886.75 \ > \ 0. \end{array}
\end{equation}
 Combining \eqref{eq:polynomial-1}--\eqref{eq:polynomial-3}, we deduce that $P$ is decreasing in $(0, 1/2)$.
 To prove that $P$ also has a unique root in this interval, we use its monotonicity and continuity, the fact that
 $$ P \,(0) \ = \ 9000 \ > \ 0 \quad \hbox{and} \quad P (1/2) \ = \ -7229.375 \ < \ 0 $$
 and the intermediate value theorem.
\end{proof} \\ \\
 Theorem \ref{th:fixation}.c directly follows from Lemma \ref{lem:polynomial} and the fact that $P \,(0.2134) > 0$.
 

\end{document}